\documentclass[12pt]{amsart}

\usepackage{amssymb}
\usepackage{pgf,interval}
\usepackage{graphicx}
\usepackage[cmtip,all]{xy}
\usepackage{ytableau}
\ytableausetup{centertableaux,nobaseline}

\usepackage{mathrsfs,calligra}

\usepackage[bbgreekl]{mathbbol}
\DeclareSymbolFontAlphabet{\mathbb}{AMSb}
\DeclareSymbolFontAlphabet{\mathbbl}{bbold}

\usepackage{accents}
\usepackage{hyperref}

\numberwithin{equation}{section}

\theoremstyle{plain}
\newtheorem{theorem}{Theorem}[section]
\newtheorem{lemma}[theorem]{Lemma}
\newtheorem{proposition}[theorem]{Proposition}
\newtheorem{corollary}[theorem]{Corollary}
\newtheorem{assumption}[theorem]{Assumption}
\newtheorem{conjecture}[theorem]{Conjecture}

\theoremstyle{definition}
\newtheorem{Fact}[theorem]{Fact}

\newtheorem{definition}[theorem]{Definition}
\newtheorem{remark}[theorem]{Remark}

\newcommand{\bbB}{\mathbb{B}}
\newcommand{\bbH}{\mathbb{H}}
\newcommand{\bbK}{\mathbb{K}}

\newcommand{\bbG}{\mathbb{G}}
\newcommand{\C}{\mathbb{C}}
\newcommand{\bbP}{\mathbb{P}}
\newcommand{\bbQ}{\mathbb{Q}}
\newcommand{\bbU}{\mathbb{U}}

\newcommand{\bbX}{\mathbb{X}}
\newcommand{\bbd}{\boldsymbol{d}}
\newcommand{\bbg}{\boldsymbol{g}}
\newcommand{\bbk}{\boldsymbol{k}}
\newcommand{\bbp}{\boldsymbol{p}}
\newcommand{\bbu}{\boldsymbol{u}}
\newcommand{\bbx}{\boldsymbol{x}}

\newcommand{\lie}[1]{\mathfrak{#1}}

\newcommand{\Lie}{\mathop\mathrm{Lie}\nolimits{}}

\newcommand{\bbPhi}{\mathbbl{\Phi}\hspace{-.53em}\mathbb{I}\hspace{.24em}}

\newcounter{thmenum}
\newenvironment{thmenumerate}{%
\begin{list}{$(\thethmenum)$}{%
\usecounter{thmenum}
\setlength{\labelsep}{.5em}
\setlength{\labelwidth}{-7pt}
\setlength{\topsep}{0pt}
\setlength{\partopsep}{0pt}
\setlength{\parsep}{0pt}
\setlength{\leftmargin}{3pt}
\setlength{\rightmargin}{0pt}
\setlength{\itemindent}{\leftmargin}
\setlength{\itemsep}{0pt}
}}
{\end{list}}

\newcommand{\mycomment}[1]{} 

\newcommand{\diag}{\qopname\relax o{diag}}
\newcommand{\rank}{\qopname\relax o{rank}}

\newcommand{\Stab}{\qopname\relax o{Stab}}

\newcommand{\id}{\qopname\relax o{id}}

\newcommand{\Aut}{\qopname\relax o{Aut}}

\newcommand{\proj}{\qopname\relax o{pr}}

\renewcommand{\Im}{\qopname\relax o{Im}}

\newcommand{\Image}{\qopname\relax o{Im}}

\newcommand{\closure}[1]{\overline{#1}}

\newcommand{\transpose}[1]{\,{}^t{#1}}

\newcommand{\restrict}{\big|}

\newcommand{\Sym}{\mathop{\mathrm{Sym}}\nolimits}

\newcommand{\orbit}{\mathbb{O}}
\newcommand{\calorbit}{\mathcal{O}}

\newcommand{\GL}{\mathrm{GL}}

\newcommand{\Sp}{\mathrm{Sp}}

\newcommand{\Mat}{\mathrm{M}}
\newcommand{\regMat}{\mathrm{M}^{\circ}}
\newcommand{\regCMat}{\mathscr{S}^{\circ}}

\newcommand{\Grass}{\qopname\relax o{Gr}}
\newcommand{\LGrass}{\qopname\relax o{LGr}}

\newcommand{\nilradical}[1]{\lie{u}_{#1}}

\newcommand{\vectwo}[2]{{\renewcommand{\arraystretch}{.85}\Bigl(\begin{array}{@{\,}c@{\,}}{#1}\\ {#2}\end{array}\Bigr)}}

\newcommand{\mattwo}[4]{\Bigl(\begin{array}{@{\,}c@{\;\;}c@{\,}}{#1} & {#2} \\ {#3} & {#4} \end{array}\Bigr)}
\newcommand{\matthree}[9]{\Biggl(\begin{array}{@{\,}c@{\;\;}c@{\;\;}c@{\,}}{#1}&{#2}&{#3}\\ {#4}&{#5}&{#6}\\ {#7}&{#8}&{#9}\\ \end{array}\Biggr)}

\newcommand{\eb}{\boldsymbol{e}}

\newcommand{\FlagB}{\mathscr{B}}
\newcommand{\STab}{\mathrm{STab}}

\newcommand{\nilpotents}{\mathcal{N}}
\newcommand{\nilpotentvark}{\nilpotents^{\theta}}
\newcommand{\nilpotentvars}{\nilpotents^{-\theta}}

\newcommand{\nilpotentsg}{\nilpotents_{\lie{g}}}

\newcommand{\Nilpotents}{\mathfrak{N}}
\newcommand{\Nilpotentvark}{\Nilpotents^{\theta}}
\newcommand{\Nilpotentvars}{\Nilpotents^{-\theta}}

\newcommand{\scrdblFV}{\mathscr{X}}
\newcommand{\dblFV}{X}
\newcommand{\bbdblFV}{\mathbb{X}}

\newcommand{\conormalvariety}{\mathcal{Z}}

\newcommand{\skipover}[1]{}

\newcommand{\shape}[1]{{\qopname\relax o{shape}}(#1)}

\newcommand{\RS}{\qopname\relax o{RS}}
\newcommand{\genRS}{\qopname\relax o{gRS}}
\newcommand{\RSl}{\RS_1}
\newcommand{\RSr}{\RS_2}
\newcommand{\Rect}{\qopname\relax o{Rect}}

\newcommand{\FlagVar}{\mathscr{F}\!\ell}
\newcommand{\FlQK}{\!\mathscr{Q}_K}
\newcommand{\FlPG}{\!\mathscr{P}_G}
\newcommand{\FlBG}{\!\mathscr{B}_G}

\newcommand{\Partition}{\mathscr{P}}
\newcommand{\partitionsof}[1]{\Partition({#1})}

\newcommand{\Orbit}{\mathfrak{O}}

\newlength{\dhatheight}

\newcommand{\projk}{\proj_{\lie{k}}}
\newcommand{\projs}{\proj_{\lie{s}}}

\newcommand{\regTnxTn}{(T_n^2)^{\circ}}
\newcommand{\regCnxCn}{(\mathscr{C}_n^2)^{\circ}}

\newcommand{\SYD}{\mathrm{YD}_{\pm}}
\newcommand{\SYDCI}{\mathrm{YD}_{\pm}^{\scriptscriptstyle\mathrm{CI}}}

\newcommand{\hatT}{\widehat{T}}

\newcommand{\bbBK}{\bbB_{\bbK}}

\newcommand{\smallplus}{\scriptscriptstyle{+}}
\newcommand{\smallminus}{\scriptscriptstyle{-}}

\begin{document}

\title[Orbit embedding for Steinberg maps]{Orbit embedding for double flag varieties and\\ Steinberg maps}

\author{Lucas Fresse}
\address{Universit\'e de Lorraine, CNRS, Institut \'Elie Cartan de Lorraine, UMR 7502, Van\-doeuvre-l\`es-Nancy, F-54506, France}
\curraddr{}
\email{lucas.fresse@univ-lorraine.fr}
\thanks{L. F. is supported in part by the ANR project GeoLie ANR-15-CE40-0012.}

\author{Kyo Nishiyama}
\address{Department of Mathematics, Aoyama Gakuin University, Fuchinobe 5-10-1, Chuo, Sagamihara 229-8558, Japan}
\curraddr{}
\email{kyo@gem.aoyama.ac.jp}
\thanks{K.~N.~is supported by JSPS KAKENHI Grant Number \#{16K05070}.}



\newcommand{\version}{Ver.~0.0}
\newcommand{\setversion}[1]{\renewcommand{\version}{Ver.~{#1}}}
\setversion{0.01 [2019/11/04 16:09:52 JST]}
\setversion{0.1 [2019/11/18 22:09:06 JST]}
\setversion{0.2 [2019/11/20 22:44:46 JST]}
\setversion{0.3 [2019/11/22 18:22:45 JST]}
\setversion{0.4 [2019/12/09]}
\setversion{0.5 [2019/12/13 00:00:37 JST]}
\setversion{0.6 [2019/12/22 10:35:37 JST]}
\setversion{0.7 [2019/12/24 14:44:59 JST]}
\setversion{1.0 [2019/12/28 16:14:50 JST]}
\setversion{1.1 [2020/05/31 10:56:34 JST]}
\setversion{2.0 [2020/06/04 18:44:22 JST]}
\setversion{3.0 [2020/09/04 10:53:43 JST]}


\begin{abstract}
In the first half of this article, 
we review the Steinberg theory for double flag varieties for symmetric pairs.  
For a special case of the symmetric space of type AIII, 
we will consider $ \dblFV = \GL_{2n}/P_{(n,n)} \times \GL_n / B_n^+ \times \GL_n / B_n^- $ on which 
$ K = \GL_n \times \GL_n $ acts diagonally.  
We give a classification of $ K $-orbits in $ \dblFV $, and 
explicit combinatorial description of the Steinberg maps.

In the latter half, we develop the theory of embedding of a double flag variety into a larger one.
This embedding is a powerful tool to study different types of double flag varieties in terms of the known ones.  
We prove an embedding theorem of orbits in full generality 
and give an example of type CI which is embedded into type AIII.
\end{abstract}

\maketitle


\section*{Introduction}

Various combinatorial structures play important roles in representation theory.  
For example, the set of all equivalence classes of irreducible representations of 
the symmetric group $ S_n $ of order $ n $ is classified by 
the set of partitions $ \partitionsof{n} $ of $ n $.  
If an irreducible representation $ \sigma $ of $ S_n $ is corresponding to a partition $ \lambda \in \partitionsof{n} $, 
then the set of standard tableaux $ \STab_{\lambda} $ gives a basis of the representation space of $ \sigma $.  
This is known as the Specht theory.  
Partitions are also interpreted as highest weights of finite dimensional irreducible representations of $ \GL_n $ 
and semistandard tableaux give a basis of an irreducible representation.  

In this respect, the geometry of flag varieties also interacts with combinatorics and 
representation theory.  
So let $ G $ be a reductive algebraic group over the complex number field, 
$ B\subset G $ a Borel subgroup, and consider the full flag variety $ G/ B $.  
Then a symmetric subgroup $ K $ of $ G $ which is fixed by an involution $ \theta $ acts on 
$ G/B $ with finitely many orbits and, together with the information of local systems, 
they actually classify irreducible Harish-Chandra $ (\lie{g}, K) $-modules 
with the trivial infinitesimal character in terms of $ D $-modules 
(\cite{Beilinson.Bernstein.1981}, see also \cite{Milicic.1993}).  
The combinatorics of the $ K $-orbits in $ G/B $ 
and their closure relations is deeply related to
the category of Harish-Chandra modules (see \cite{McGovern.1998,Lusztig.Vogan.1983}, e.g.).  

Let us write $ \FlagB = G/B $ for shorthand.
In much earlier time, Springer noticed that 
the cotangent bundle $ T^* \FlagB $ gives a resolution of singularity 
of the nilpotent variety $ \nilpotentsg $ consisting of the nilpotent elements in $ \lie{g} $ 
(\cite{Springer.1969}).  
He constructed irreducible representations of the Weyl group $ W $ on 
the top-degree cohomology space of the so-called Springer fiber $ \FlagB_x $ associated 
to a nilpotent element $ x \in \nilpotentsg $, 
and in that way 
he created a correspondence between irreducible representations of 
$ W $ and the set of nilpotent orbits $ \nilpotentsg / G $ together with their local systems, which is finite in number 
(\cite{Springer.1978}, see also \cite{Chriss.Ginzburg.1997}).  
This is a remarkable breakthrough that provides an amount of works on 
combinatorics related to the geometry of flag varieties as well as the nilpotent varieties 
(see, e.g., \cite{Kazhdan.Lusztig.1979,Richardson.Springer.1993} and \cite{Trapa.2005,NTW.2016}).  

Approximately in the same period, 
Steinberg introduced a variety, now called the Steinberg variety 
(\cite{Steinberg.1976}).  
He used this variety to study the Springer resolution more deeply.  
The resolution 
$ T^* \FlagB \to \nilpotentsg $ actually comes as 
a moment map arising from the Hamiltonian action of $ G $ on the symplectic variety $ T^* \FlagB $
(see \cite{Chriss.Ginzburg.1997,Douglass.Roehrle.2009}).  
Let us consider the product of the flag variety $ \scrdblFV := \FlagB \times \FlagB $ on which 
$ G $ acts diagonally.  
It is traditional to study a symplectic reduction in the study of Hamiltonian actions on a symplectic variety, and the Steinberg variety is obtained from this recipe.  
It arises as the null fiber of the moment map $ \mu_\scrdblFV : T^* \scrdblFV \to \lie{g}^* $, and is 
denoted by $ \conormalvariety_\scrdblFV = \mu_\scrdblFV^{-1}(0) $.  
The Steinberg variety can also be interpreted as the fiber product over the resolution maps, 
which is exhibited below.
\begin{equation*} 
\xymatrix @R+1ex @M+.5ex @C-2ex @L+.5ex {
\conormalvariety_\scrdblFV {=} T^*\FlagB {\times_{\nilpotentsg}} T^*\FlagB \ar[d]^-{p_1} \ar[rr]^-{p_2}  \ar@{->}[drr]|-{\;\varphi\;}
& & T^*\FlagB \ar[d]^{\mu_{\FlagB}} \\
T^*\FlagB \ar[rr]_-{\mu_{\FlagB}} & & \nilpotentsg
}
\end{equation*}
The variety $ \conormalvariety_\scrdblFV $ is highly reducible of equi-dimension, and 
its irreducible components are parametrized by the Weyl group $ W $, or strictly speaking 
by the $ G $-orbits $ \{ \orbit_w \mid w \in W \} $ in $ \scrdblFV $.  
The image of an irreducible component by the map $ \varphi $, which is the composite of 
the first projection to $ T^* \FlagB $ and then by the moment map 
$ \mu_{\FlagB} : T^* \FlagB \to \nilpotentsg $, is the closure of a nilpotent orbit.  
\begin{equation*}
\xymatrix @R+1ex @M+.5ex @C-2ex @L+.5ex {
& T^*\FlagB {\times_{\nilpotentsg}} T^*\FlagB \ar[ld]_-{\pi} \ar[dr]^-{\;\varphi\;} \\
\makebox[3ex][c]{$\FlagB{\times}\FlagB \supset \orbit_w $} & & \makebox[3ex][c]{$\nilpotentsg \supset \calorbit $}
} 
\end{equation*}
Thus we get a combinatorial map 
$ {\Phi} : W \simeq (\FlagB{\times}\FlagB)/G \to \nilpotentsg/ G $ 
via
$ \varphi\bigl( \closure{\pi^{-1}(\orbit_w)} \bigr) = \closure{\calorbit} \; ( w \in W,\, \calorbit \in \nilpotentsg / G) $.  
This provides a rich theory involving geometry, combinatorics and representation theory.  
In fact, it turns out that $ \conormalvariety_\scrdblFV $ bears the regular representation of the Weyl group 
which ``doubles'' the Springer representations.

We started a study which generalizes the above mentioned Steinberg theory 
to the case of symmetric pairs in \cite{Fresse.N.2016,Fresse.N.2020}.  
It is related to the triple flag varieties \cite{MWZ.1999,MWZ.2000,FGT.2009,Travkin.2009}  
as well as the double flag varieties for symmetric pairs \cite{NO.2011,HNOO.2013} of finite type.  
In fact, our study is motivated by the paper by Henderson and Trapa \cite{Henderson.Trapa.2012},  
although the role of the double flag variety is implicit in their paper.

Let $ P $ and $ Q$ be parabolic subgroups of $ G $ and $ K $, respectively.  
Then we call $ \dblFV := G/P \times K/Q $ a \emph{double flag variety} of the symmetric pair $ (G, K) $.  
The symmetric subgroup $ K $ acts on $ \dblFV $ diagonally.  
For this double flag variety, we define a conormal variety 
$ \conormalvariety_{\dblFV} $, 
which is a direct generalization of the Steinberg variety.  
If there are only finitely many $ K $-orbits on $ \dblFV $, 
basically we can play the same game as above, and obtain 
\emph{two} different combinatorial maps 
$ \Phi^{\pm \theta} : \dblFV / K \to \nilpotents^{\pm \theta} / K $, 
where $ \nilpotentvark = \nilpotentsg \cap \lie{k} $ is the nilpotent variety of $ \lie{k} $, 
and $ \nilpotentvars $ is the nilpotent variety in the (co)tangent space of $ G/K $ (see below for the precise definition).  

We call the maps $ \Phi^{\pm \theta} $ defined above 
\emph{Steinberg maps} associated to $ \dblFV $.  
If we need to distinguish $ \Phi^\theta $ and $ \Phi^{-\theta} $, 
the map $ \Phi^{\theta} $ is called a \emph{generalized} Steinberg map and 
$ \Phi^{-\theta} $ an \emph{exotic} one.

In the case of the double flag varieties, it is not obvious if there is a representation theoretic structure on 
$ \conormalvariety_X $.  
However, a na\"{i}ve picture of convolutions (cf.~\cite[\S 2.7]{Chriss.Ginzburg.1997}) gives us an insight that hopefully 
we can construct Hecke algebra actions of $ \mathcal{H}(G/B) $ on the left and 
those of $ \mathcal{H}(K/B_K) $ on the right and which make the top Borel-Moore homology space into 
a Hecke algebra bi-module.  
See \cite{Travkin.2009} and also \cite[Conjecture 7.11]{Fresse.N.2020}.  

Let us consider an example of a symmetric pair of type AIII, namely 
$ (G, K) = (\GL_{2n}, \linebreak[3] \GL_n \times \GL_n) $.  
We consider everything over $ \C $ and basically omit the letter $ \C $.  
Let $ V = \C^{2n} $ and fix
a standard polar decomposition $ V = V^+ \oplus V^- $, 
where $ \dim V^{\pm} = n $ and 
$ V^+ = \langle \eb_i \mid 1 \leq i \leq n \rangle $,  
$ V^- = \langle \eb_i \mid n + 1 \leq i \leq 2 n \rangle $.  
Then $ K $ is the stabilizer of the polarization.  
Consider a maximal parabolic subgroup $ P $ in $ G $ which stabilizes $ V^+ $ so that 
$ G/P \simeq \Grass_n(V) $, the Grassmannian of $ n $-spaces in $ V $.  
Also we choose a Borel subgroup of $ K = \GL_n \times \GL_n $ as 
$ B_K = B_n^+ \times B_n^- $, where 
$ B_n^+ $ denotes the Borel subgroup of $ \GL_n $ consisting of upper triangular matrices and 
$ B_n^- $ its opposite as usual.  
In this way, our double flag variety is 
\begin{equation*}
X = G/P \times K/B_K \simeq \Grass_n(V) \times \FlagVar_n^+ \times \FlagVar_n^- ,
\end{equation*}
on which $ K = \GL_n^2 $ acts.  Here $ \FlagVar_n^{\pm} $ is the set of complete flags of subspaces in $ V^{\pm} $.
We proved 

\begin{theorem}[\cite{Fresse.N.2020}, see Theorem~\ref{thm:orbits.in.dblFV.typeAIII} below]
There are finitely many $ K $-orbits in $ \dblFV $ and 
they are parametrized by pairs of partial permutations of rank $ n $.  
Namely, we have 
\begin{equation*}
X/K \simeq \regTnxTn / S_n , \qquad
\regTnxTn := \Big\{ 
\omega = \vectwo{\tau_1}{\tau_2} \Bigm| \tau_1, \tau_2 \in T_n , \rank \omega = n \Big\} , 
\end{equation*}
where $ T_n $ denotes the set of partial permutation matrices of size $ n $.
\end{theorem}

It is well known that the nilpotent orbits in
$ \nilpotentvark / K $ are parametrized by pairs of partitions $ (\lambda, \mu) \in \partitionsof{n}^2 $, 
and those in 
$ \nilpotentvars / K $ are parametrized by signed Young diagrams $ \Lambda \in \SYD(n,n) $ of signature $ (n, n) $.  
See \S~\ref{subsec:symmetric.pair.typeAIII} below for details.  
So we obtain a combinatorial map 
\begin{equation*}
\Phi^{\pm \theta} : 
\regTnxTn /S_n \simeq X/K \xrightarrow{\qquad} \nilpotents^{\pm \theta}/K \simeq \begin{cases}
\partitionsof{n}^2 & \text{ for $ \nilpotentvark $}, 
\\
\SYD(n,n) & \text{ for $ \nilpotentvars $}.
\end{cases}
\end{equation*}

In \cite{Fresse.N.2020}, we gave partial results which give 
explicit and efficient algorithms for computing the Steinberg maps $ \Phi^{\pm \theta} $ 
for $ \omega = \vectwo{\tau_1}{1_n} $, i.e., one of the partial permutation is really a permutation.  
However, now we have a complete algorithm for all of $ \regTnxTn /S_n $ (unpublished, in preparation).
In the present paper, we give this complete algorithm without proof for the map $ \Phi^{\theta} $.  
Proofs and the algorithm for $ \Phi^{-\theta} $ will appear elsewhere soon. 
Thus we only claim it in an abstract manner here. 

\begin{theorem}[Theorems~\ref{thm:existence.of.algorithm} and \ref{thm:gen.Steinberg.map}]
There exist efficient combinatorial algorithms which describe the maps 
$ \Phi^{\pm \theta} $.  
\end{theorem}

In the course of proofs, we obtain a generalization of 
Robinson-Schensted correspondence 
for the pairs of partial permutations of full rank 
(see Theorem~\ref{thm:gen.Steinberg.map}).  
This correspondence is also interesting in itself, and we noticed\footnote{
Actually this was pointed out to the authors by Anthony Henderson (private communication).  
We thank him for that.}
that there is a strong resemblance of the parameter sets of 
$ K $-orbits in $ X $ and 
those for the Travkin's mirabolic triple flag variety (\cite{Travkin.2009}).  
It seems that they are also related to the parameter sets of 
Achar-Henderson's enhanced nilpotent orbits (\cite{Achar.Henderson.2008}) and 
Kato's exotic nilpotent orbits (\cite{Kato.2009}).  
However, up to now, we cannot make out any rigorous geometric interpretations for them.  

There are many other interesting double flag varieties which admit a finite number of $ K $-orbits 
(see \cite{NO.2011,HNOO.2013}).   
However, no unified way to get an explicit and efficient algorithm for 
the above mentioned Steinberg maps $ \Phi^{\pm \theta} $ is known up to now.  
Even for giving parametrizations of $ K $-orbits in $ \dblFV $, there is no rigorous theory.

In this paper, we propose a technique by which 
we can embed a double flag variety into a larger double flag variety 
preserving orbit structures.  
Thus, if we know the Steinberg theory for a larger double flag variety, we might deduce 
the theory for a smaller one.  

In fact, nilpotent orbits of classical Lie algebras are classified in that way.  
Namely, we first classify the nilpotent orbits in type A, which amounts to establish the theory of 
Jordan normal form, and we get partitions.  
For type B, C and D, we embed them as Lie subalgebras of type A.   
Then nilpotent orbits of these Lie subalgebras can be obtained 
as non-empty intersection of the ones in type A and these subalgebras themselves.  
So we can use partitions of special shapes as a parameter set of nilpotent orbits.

In Section~\ref{sec:abstract.embedding.theory}, we discuss 
the embedding of orbits in $ \dblFV $ into a larger double flag variety $ \bbdblFV $.  
The key idea is to use two commuting involutions $ \sigma, \theta $ and 
existence of square roots in the direction of $ G/K $.  
This idea is originally developed by Takuya Ohta \cite{Ohta.2008} for linear actions, 
and later for arbitrary actions in \cite{Nishiyama.2014}.  

We give a full general theory for embedding in Section~\ref{sec:abstract.embedding.theory} 
without assuming the finiteness of orbits.  
Then, in Section~\ref{section:embedding.typeCI.into.typeAIII},  we give an example of type CI embedded into type AIII discussed above.  

Let us briefly summarize the main results here.  
Thus we will consider a larger connected reductive algebraic group $ \bbG $ and 
two commuting involutions $ \sigma, \theta $ of $ \bbG $.  
Define $ \bbK = \bbG^{\theta} $, the fixed point subgroup of $ \theta $, and similarly 
$ G = \bbG^{\sigma} $, $ K = \bbK^{\sigma} = G^{\theta} $.  
We assume that all these groups are connected.  
Take parabolic subgroups $ P \subset G $ and $ Q \subset K $.  
Then there exist $ \sigma $-stable parabolic subgroups $ \bbP \subset \bbG $ and 
$ \bbQ \subset \bbK $ which cut out $ P $ and $ Q $ from $ G $ and $ K $ respectively.  

Take a $ \sigma $-stable subgroup $ \bbH $ of $ \bbG $ and 
denote $ \bbH^{-\sigma} = \{ h \in \bbH \mid \sigma(h) = h^{-1} \} $.  
We say that \emph{$ \bbH $ admits $ (-\sigma) $-square roots} if for any $ h \in \bbH^{-\sigma} $ 
there exists an $ f \in \bbH^{-\sigma} $ such that $ h = f^2 $.
We consider the following conditions.  
\begin{itemize}
\item[(A)] 
$ \bbP $ and $ \bbQ $ admit $ (-\sigma) $-square roots.

\item[(B)]
For any $ \sigma $-stable parabolic subgroups $ \bbP_1 \subset \bbG $ and 
$ \bbQ_1 \subset \bbK $ which are conjugate to $ \bbP $ and $ \bbQ $ respectively, 
the intersection $ \bbP_1 \cap \bbQ_1 $ admits $ (-\sigma) $-square roots.
\end{itemize}

\begin{theorem}[Theorem~\ref{thm:embedding.orbits}]
In the above setting, 
let us consider the double flag varieties 
$ \dblFV := G/P \times K/Q $ and $ \bbdblFV := \bbG / \bbP \times \bbK / \bbQ $.  
If the parabolic subgroups $ \bbP $ and $ \bbQ $ satisfy the conditions {\upshape(A)} and {\upshape(B)}, then
there exists a natural embedding $ \dblFV \hookrightarrow \bbdblFV $ which respects the involution $ \sigma $, 
and the natural orbit map 
$ \iota : \dblFV / K \to \bbdblFV / \bbK $ defined by 
$ \iota(O) = \bbK \cdot O $  for $ O \in \dblFV / K $ is injective, 
i.e., for any $ \bbK $-orbit $ \orbit $ in $ \bbdblFV $, 
the intersection $ \orbit \cap \dblFV $ is either empty or a single $ K $-orbit.
\end{theorem}

Under the same assumption, we prove a theorem on conormal varieties, 
which essentially says that 
there exists an embedding of irreducible components of 
$ \conormalvariety_{\dblFV} $ into those of $ \conormalvariety_{\bbdblFV} $.  
See Theorem~\ref{thm:embedding.conormal.bundles}.

In Section~\ref{section:embedding.typeCI.into.typeAIII}, 
we consider 
$ (\bbG, \bbK) = (\GL_{2n}, \GL_n \times \GL_n) $ of type AIII and 
take $ (G, K) = (\Sp_{2n}, \GL_n) $.  
For a Siegel parabolic subgroup $ P = P_S $ in $ G $ and 
the standard upper triangular Borel subgroup $ B_K = B_n^+ $ in $ K $, 
we can choose $ \sigma $-stable parabolic subgroups $ \bbP $ and $ \bbBK $ 
which satisfy the assumptions (A) and (B).  
Thus the embedding theorem holds, and we get an explicit parametrization of 
$ K $-orbits in $ X = \Sp_{2n}/P_S \times \GL_n/B_n^+ $ (see Theorem~\ref{thm:explicit.embedding.typeCI.to.typeAIII}).  
Note that $ X \simeq \LGrass(\C^{2n}) \times \FlagVar_n $, 
where $ \LGrass(\C^{2n}) $ is the Lagrangian Grassmannian in the symplectic vector space $ \C^{2n} $.

\bigskip

We thank the organizers of the conference held in Dubrovnik, Croatia, June 24--29, 2019.  
This conference offered the authors a good opportunity to review their previous joint works seriously  
and it leads to new results reported here.  
Also we thank Anthony Henderson for the correspondence and 
Takuya Ohta for discussions on orbit embeddings.


\section{Steinberg theory for symmetric pairs: a review}\label{S1:Steinberg.theory.for.symmetric.pairs}

In this section, we review the Steinberg theory for symmetric pairs given in \cite{Fresse.N.2020}, 
although we describe it in a slightly different manner in this article.

Let $ G $ be a complex connected reductive algebraic group with an involutive automorphism $ \theta: G \to G $. 
Let $ K := G^\theta $ be the fixed-point subgroup of $ \theta $.  
Thus we have a symmetric pair $ (G, K) $ and $ K $ is called a symmetric subgroup.
Assume for simplicity that $ K $ is connected.

By differentiation, the involution $ \theta $ induces an involution on the Lie algebra $ \lie{g} := \Lie(G) $, which we also denote $ \theta: \lie{g} \to \lie{g} $ by abuse of notation. 
Let $ \lie{k} := \Lie(K) = \lie{g}^\theta $ the fixed-point subalgebra and 
put $ \lie{s} := \lie{g}^{-\theta} = \{ x \in \lie{g} \mid \theta(x) = - x \} $.  
Let $ x \mapsto x^\theta $ and $ x \mapsto x^{-\theta} $ stand for 
the projections $ \projk : \lie{g} \to \lie{k} $ and $ \projs : \lie{g} \to \lie{s} $ 
along the Cartan decomposition $ \lie{g} = \lie{k} \oplus \lie{s} $. 
Let $ \nilpotents \subset \lie{g} $ be the cone of nilpotent elements and put 
$ \nilpotentvark := \nilpotents \cap \lie{k} $ and $ \nilpotentvars := \nilpotents \cap \lie{s} $. 
It is well known that 
the nilpotent varieties $\nilpotentvark $ and $\nilpotentvars $ consist of finitely many $K$-orbits.

Let us introduce the double flag variety
\begin{math}
\dblFV = G/P \times K/Q ,
\end{math}
where $ P $ and $ Q $ are parabolic subgroups of $ G $ and $ K $ respectively 
\cite{NO.2011}.  
The variety $ \dblFV $ is a smooth projective variety on which $ K $ acts diagonally.  
Let us denote $ \FlPG = G/P $ and $ \FlQK = K/Q $ so that 
\begin{equation*}
\dblFV = \FlPG \times \FlQK .
\end{equation*}

As usual, 
we identify $ \FlPG = G/P $ with the set of parabolic subalgebras $ \lie{p}_1 $ which are conjugate to 
$ \lie{p} = \Lie(P) $.  We denote by $ \nilradical{\lie{p}_1} $ 
the nilpotent radical of a parabolic subalgebra $ \lie{p}_1 $.  
Then the cotangent bundle over $ \FlPG $ is isomorphic to 
\begin{equation*}
T^*\FlPG = \{ ( \lie{p}_1, x ) \mid \lie{p}_1 \in \FlPG, \; x \in \nilradical{\lie{p}_1} \} 
\simeq G \times_P \nilradical{\lie{p}} . 
\end{equation*}
We denote by $ \mu_{\FlPG} : T^*\FlPG \to \nilpotents $ the second projection 
$ \mu_{\FlPG}(\lie{p}_1, x) = x $, which coincides with the moment map\footnote{
The cotangent bundle $ T^*\FlPG $ admits a canonical $ G $-invariant symplectic structure.  
Since the action of $ G $ on the symplectic variety 
$ T^*\FlPG (\simeq G \times_P \nilradical{\lie{p}}) $ is Hamiltonian, 
there exists a moment map $ \mu : T^*\FlPG \to \lie{g}^* $.  
We fix once and for all a nondegenerate invariant bilinear form on $ \lie{g} $ and identify 
$ \lie{g} \simeq \lie{g}^* $.  
With this identification, the conormal direction 
$ (\lie{g}/\lie{p}_1)^\bot \subset \lie{g}^* $ is identified with $ \lie{u}_{\lie{p}_1} $ 
which is contained in $ \nilpotents $.  So the image of $ \mu $ is actually contained 
in the nilpotent variety $ \nilpotents $.  
We do not repeat similar arguments below, but the term ``moment map'' should be always understood in this way.
} 
with respect to 
a standard symplectic structure on $ T^*\FlPG $.  
Similarly, we have the moment map 
\begin{equation*}
\mu_{\FlQK} : T^*\FlQK = \{ ( \lie{q}_1, y ) \mid \lie{q}_1 \in \FlQK, \; y \in \nilradical{\lie{q}_1} \} \to \lie{k} , \quad 
\mu_{\FlQK}(\lie{q}_1, y) = y ,
\end{equation*}
with the obvious notation similar to those for $ \FlPG $.

\begin{definition}
Let $ \conormalvariety := T^*\FlPG \times_{\nilpotentvark} T^*\FlQK $ be the fiber product over the nilpotent variety $ \nilpotentvark $:
\begin{equation*}
\xymatrix @R-1ex @M+.5ex @C-2ex @L+.5ex {
\conormalvariety = T^*\FlPG {\times_{\nilpotentvark}} T^*\FlQK
\ar[dd]^-{p_2} \ar[rr]^-{p_1}  \ar@{->}[ddrr]|-{\;\;\varphi^{\theta}\;\;}
& & T^*\FlPG \ar[d]^{\mu_{\FlPG}} \ni (\lie{p}_1, x) \\
& & \nilpotents_{\lie{g}} \ni x \ar[d]^{\projk} \\
(\lie{q}_1, y) \in T^*\FlQK \ar[rr]^-{-\mu_{\FlQK}} & & \nilpotentvark \ni -y = x^{\theta} 
} 
\end{equation*}
We call $ \conormalvariety = \conormalvariety_{\dblFV} $ the \emph{conormal variety} for the double flag variety 
$ \dblFV $.
\end{definition}

The definition of the conormal variety $ \conormalvariety $ looks different from 
that in \cite{Fresse.N.2020}, but they are isomorphic.  
In fact, we know

\begin{Fact}\label{fact:conormal.var.is.null.fiber}
\begin{thmenumerate}
\item\label{fact:conormal.var.is.null.fiber:item:1}
Let $ \mu_{\dblFV} : T^*\dblFV \to \lie{k} $ be the moment map 
for the canonical Hamiltonian action of $ K $ on the cotangent bundle $ T^*\dblFV $.  
Then the conormal variety is isomorphic to the null fiber of the moment map: 
$ \conormalvariety \simeq \mu_{\dblFV}^{-1}(0) $.

\item
Let $ \orbit \subset \dblFV $ be a $ K $-orbit.  
We denote by $ T^*_{\orbit}\dblFV $ the conormal bundle over $ \orbit $.  
The conormal variety is a disjoint union of the conormal bundles: 
$ \conormalvariety = \coprod_{\orbit \in \dblFV/K} T^*_{\orbit}\dblFV $.

\item\label{fact:conormal.var.is.null.fiber:item:3}
The dimension of the conormal variety is equal to $ \dim \dblFV $ 
if and only if there are only finitely many $ K $-orbits in $ \dblFV $.  
In this case, $ \conormalvariety $ is equidimensional and 
each irreducible component arises as the closure of a conormal bundle.  
Thus 
$ \conormalvariety = \bigcup_{\orbit \in \dblFV/K} \closure{T^*_{\orbit}\dblFV} $ 
gives the decomposition into irreducible components.

\end{thmenumerate}
\end{Fact}

We are particularly interested in the case where 
there are only finitely many $ K $-orbits in $ \dblFV $.  
However, for the time being, we do not assume it and develop a general theory.

Let us denote the diagonal map in the fiber product by 
$ \varphi^{\theta} : \conormalvariety \to \nilpotentvark $.  
This map is explicitly described as 
\begin{equation*}
\varphi^{\theta}( (\lie{p}_1, x), (\lie{q}_1, y) ) = x^{\theta} = - y \quad
\text{ for } \;\;
( (\lie{p}_1, x), (\lie{q}_1, y) ) \in \conormalvariety .
\end{equation*}
Note that we consider 
$ \conormalvariety \subset T^*\FlPG {\times} T^*\FlQK = T^*\dblFV $ here.  
It is not enough to specify the conormal fiber only by $ \varphi^{\theta} $ and 
we need another map 
\begin{equation*}
\varphi^{-\theta}( (\lie{p}_1, x), (\lie{q}_1, y) ) = x^{-\theta} = x + y \quad
\text{ for } \;\;
( (\lie{p}_1, x), (\lie{q}_1, y) ) \in \conormalvariety .
\end{equation*}
We call $ \varphi^{\theta} $ the \emph{generalized Steinberg map} and 
$ \varphi^{-\theta} $ the \emph{exotic Steinberg map}.  
Both maps are clearly $ K $-equivariant, but a priori not closed (see \cite[Remark 11.3]{Fresse.N.2020}).

By definition, 
the image $ \Image \varphi^{\theta} $ is contained in the nilpotent variety 
$ \nilpotentvark $.  
There is no guarantee that the image $ \Image \varphi^{-\theta} $ is contained 
in the nilpotent variety, 
but in many interesting cases it is so.  
For this, we refer the readers to \cite[\S 4.1]{Fresse.N.2016}.

\begin{assumption}\label{assumption:exotic.Steinberg.map.to.nilpotents}
We assume $ \Image \varphi^{-\theta} \subset \nilpotentvars $ throughout in this paper.  
\end{assumption}

Note that 
the sets $G\nilradical{\lie{p}}\subset\nilpotents $ and 
$K\nilradical{\lie{q}}\subset\nilpotentvark $ are the closures of 
the Richardson nilpotent orbits associated to $P$ and $Q$ respectively.  
So the above assumption is equivalent to claiming that 
$ \projs(G\nilradical{\lie{p}}\cap(K\nilradical{\lie{q}}+\lie{s} )) \subset \nilpotentvars $. 
In other words,
$ x \in G\nilradical{\lie{p}} $ and $ x^\theta \in K\nilradical{\lie{q}} $ imply 
$ x^{-\theta} \in \nilpotentvars $.

Let $ \pi $ be the projection from the cotangent bundle $ T^*\dblFV $ to $ \dblFV $ (the bundle map) 
and consider the following double fibration maps.
\begin{equation*} 
\xymatrix @R+1ex @M+.5ex @C-2ex @L+.5ex {
& \conormalvariety = T^*\FlQK {\times_{\nilpotents^\theta}} T^*\FlPG \ar[ld]_-{\pi} \ar[dr]^-{\;\varphi^{\pm \theta}\;} \\
\makebox[3ex][c]{$ \dblFV = \FlQK\times\FlPG $} & & \makebox[3ex][c]{$\nilpotents^{\pm \theta}$}
} 
\end{equation*}
Using this diagram, we define orbit maps 
\begin{equation*}
\Phi^{\pm \theta} : \dblFV/K \xrightarrow{\qquad} \nilpotents^{\pm \theta}/K, 
\qquad
\orbit \mapsto \calorbit 
\end{equation*}
by 
$ \closure{\varphi^{\pm \theta}( \pi^{-1}(\orbit) )} 
= \closure{\calorbit} $, 
where $ \orbit $ is a $ K $-orbit in $ \dblFV $ and 
$ \calorbit $ is a nilpotent $ K $-orbit in $ \nilpotents^{\pm \theta} $.  
This definition works since there are only finitely many nilpotent $ K $-orbits 
both in $ \nilpotentvark $ and $ \nilpotentvars $.
By abuse of the terminology, we also call 
$ \Phi^{\theta} $ the \emph{generalized Steinberg map} and 
$ \Phi^{-\theta} $ the \emph{exotic Steinberg map}.  

Since $ \pi^{-1}(\orbit) = T^*_{\orbit}\dblFV $ is the conormal bundle over $ \orbit $, 
$ \Phi^{\pm \theta}(\orbit) = \calorbit $ if and only if 
$ \closure{\varphi^{\pm \theta}(T^*_{\orbit}\dblFV)} = \closure{\calorbit} $.  
Recall that, 
if there exist only finitely many $ K $-orbits in $ \dblFV $, the closure of
$ T^*_{\orbit}\dblFV $ can also be interpreted as 
an irreducible component of the conormal variety $ \conormalvariety $
(Fact~\ref{fact:conormal.var.is.null.fiber}~\eqref{fact:conormal.var.is.null.fiber:item:3}).

\section{Combinatorial Steinberg maps}

In this section, 
we will discuss a combinatorial side of the Steinberg theory.  
Thus \emph{we assume that there are finitely many $ K $-orbits in the double flag variety $ \dblFV = \FlPG \times \FlQK $}.  
In this situation,  
if we know an explicit classification of the $ K $-orbits in $ \dblFV $, 
both Steinberg maps
$ \Phi^{\pm \theta} : \dblFV / K \to \nilpotents^{\pm \theta} $ 
might have interesting combinatorial interpretations.  

\subsection{Classical Steinberg map and the Robinson-Schensted correspondence}

Let us first review the results by Steinberg \cite{Steinberg.1988}.
We consider a special case where $ G = K $ 
(i.e., $ \theta = \id_G $), and take $ P = Q = B $ to be a Borel subgroup of $ G $.  
Then $ \dblFV = G/B \times G/B = \FlBG \times \FlBG $, 
where $ \FlBG = G/B $ is the full flag variety.  

In this case, we see that
$ \dblFV / G \simeq B \backslash G / B $, 
and the double coset space on the right hand side is parametrized by the Weyl group $ W $ 
thanks to the Bruhat decomposition.  
Also, since $ \lie{s} = 0 $, the nilpotent variety $ \nilpotentvars $ vanishes and 
$ \varphi^{-\theta} $ is zero.  
The nilpotent variety $ \nilpotentvark $ coincides with $ \nilpotents = \nilpotents_{\lie{g}} $.  
Thus our map $ \Phi^{\theta} $ reduces to $ \Phi : W \to \nilpotents / G $.  
Note that both the Weyl group $ W $ and the set of nilpotent orbits $ \nilpotents / G $ 
provide rich ingredients for combinatorics.  

Let us examine it in the case of $ G = \GL_n = \GL_n(\C) $.  
We choose 
$ B = B_n^+ $ (the Borel subgroup of upper triangular matrices).   
The Weyl group $ W $ is simply the symmetric group $ S_n $ of order $ n $ and   
the set of nilpotent $ \GL_n $-orbits in $ \nilpotents $ is 
in bijection with the set of partitions of $ n $, which we will denote by $ \partitionsof{n} $. 

For $ w \in S_n $, we take a permutation matrix denoted by the same letter, 
and write $ \orbit_w $ for the $ G $-orbit through $ (B, wB) \in \dblFV $.
Then, as we explained, 
$ \conormalvariety_w := \closure{\pi^{-1}(\orbit_w)} $ is an irreducible component of the variety
$ \conormalvariety $ (called \emph{Steinberg variety}, in this setting).  
The image of this irreducible component by $ \varphi = \varphi^{\theta} $ is 
the closure of a nilpotent orbit in $ \nilpotents $, which is parametrized by a partition 
$ \lambda \in \partitionsof{n} $.
Thus we get $ \varphi(\conormalvariety_w) = \closure{\calorbit_{\lambda}} $, 
which establishes the map $ \Phi : S_n \to \partitionsof{n} $.

It is the Robinson-Schensted correspondence that plays another important role of the theory, 
which establishes a (combinatorial) bijection 
between $ S_n $ and the set of pairs of standard tableaux of same shape (see \cite{Fulton}, for example).  
Thus we have 
\begin{equation*}
\RS : S_n \xrightarrow{\;\;\sim\;\;} 
\coprod\limits_{\lambda \in \partitionsof{n}} 
\{ (T_1, T_2) \mid T_i \in \STab_{\lambda} \} , 
\end{equation*}
where 
$ \STab_{\lambda} $ 
denotes the set of standard tableaux of the shape $ \lambda $.

\begin{theorem}[Steinberg~\cite{Steinberg.1988}]
The Steinberg map $ \Phi : S_n \ni w \mapsto \lambda \in \partitionsof{n} $ \quad
defined by $ \varphi(\conormalvariety_w) = \closure{\calorbit_{\lambda}} $ 
factors through the Robinson-Schensted correspondence.
\begin{equation*}
\xymatrix  @M 3pt @R 8ex{
S_n \ar[drr]_{\Phi} \ar@{->}[rr]^-{\sim}_-{\RS} & & \coprod\limits_{\lambda \in \partitionsof{n}} 
\{ (T_1, T_2) \mid T_i \in \STab_{\lambda} \} \ar[d] & \makebox[0ex][c]{$\ni (T_1, T_2)$} \ar@{|->}[d]\\
& & \partitionsof{n} & \makebox[0ex][c]{$\ni \lambda = \shape{T_i}$} }
\end{equation*}
\end{theorem}

\subsection{A symmetric pair of type AIII}\label{subsec:symmetric.pair.typeAIII}

Let us consider the case of the symmetric pair $(G,K)=(\GL_{2n},\GL_n\times \GL_n)$.   
This case is studied in \cite{Fresse.N.2020}. 

We take a Siegel parabolic subgroup $P=P_\mathrm{S}\subset \GL_{2n}$ of $ G $, 
which is the stabilizer of the $n$-dimensional subspace $\C^n\times\{0\}^n\subset \C^{2n}$, 
and a Borel subgroup $Q=B_n^+\times B_n^-\subset \GL_n\times \GL_n$ of $ K $, 
where $ B_n^- $ denotes the lower triangular Borel subgroup of $ \GL_n $.     
Thus the double flag variety $\dblFV $ becomes 
\[
\dblFV =\GL_n/B_n^+ \times \GL_n/B_n^- \times \Grass_n(\C^{2n}),
\]
where $\Grass_n(\C^{2n})$ stands for the Grassmann variety of $n$-dimensional subspaces in $\C^{2n}$.
In \cite[Theorem~8.1]{Fresse.N.2020}, the $ K $-orbits on $ \dblFV $ are completely classified.  
As a result, there are only finitely many $ K $-orbits in $ \dblFV $.  
Let us review the classification briefly.

A \emph{partial permutation} $ \tau $ on the set $ [n] = \{ 1, 2, \dots, n \} $ is an injective map from a subset 
$ J \subset [n] $ to $ [n] $.  
It is convenient to extend $ \tau : J \to [n] $ to 
$ \tau : [n] \to [n] \cup \{ 0 \} $ by putting 
$ \tau(k) = 0 $ for $ k \not\in J $.   
As in the case of permutation, we can associate a matrix in 
$ \Mat_n $ with a partial permutation $ \tau $, which we also denote by $ \tau $  by abuse of notation. 
Namely the matrix $ \tau $ is given by 
$ (\eb_{\tau(1)}, \eb_{\tau(2)}, \dots, \eb_{\tau(n)}) $, 
where $ \eb_0 = 0 $ and $ \eb_1, \dots, \eb_n $ denote the elementary basis vectors of $ \C^n $.
Let us denote by $ T_n $ the set of all partial permutation matrices and put 
\begin{equation*}
\regTnxTn = \Bigl\{ \omega = \vectwo{\tau_1}{\tau_2} \Bigm| \tau_1, \tau_2 \in T_n \text{ and } \rank \omega = n \Bigr\} \subset \Mat_{2n, n} .
\end{equation*}
Then the image $ [\omega] := \Im \omega $ generated by the column vectors of $ \omega $ is an $ n $-dimensional vector space, hence 
it represents a point in $ \Grass_n(\C^{2n}) $.  
Notice that the symmetric group $ S_n $ acts on $ \regTnxTn $ by the right multiplication (or permutation of the column vectors) 
and this action does not change the corresponding subspace $ [\omega] $.  
Let us denote the parabolic subgroup of $ G $ stabilizing $ [\omega] $ by $ P_{\omega} $.

\begin{theorem}[{\cite[Theorem~8.1]{Fresse.N.2020}}]\label{thm:orbits.in.dblFV.typeAIII}
There exist natural bijections 
\begin{equation*}
\dblFV /K \simeq \Grass_n(\C^{2n}) / (B_n^+\times B_n^-) \simeq \regTnxTn / S_n .
\end{equation*}
The bijections are explicitly given by 
$ K \cdot (B_n^+\times B_n^-, P_{\omega}) \leftrightarrow (B_n^+\times B_n^-) \cdot [\omega] \leftrightarrow \omega S_n $.
\end{theorem}

In this setting Assumption~\ref{assumption:exotic.Steinberg.map.to.nilpotents} is satisfied (see \cite[Table~3]{NO.2011} and \cite[Proposition~4.2]{Fresse.N.2016}).  
Thus by the machinery which we have already described, we get the Steinberg maps 
$ \Phi^{\theta} : \regTnxTn/S_n \to \nilpotentvark / K $ and 
$ \Phi^{-\theta} : \regTnxTn/S_n \to \nilpotentvars / K $.
Since $ K = \GL_n \times \GL_n $, the nilpotent variety $ \nilpotentvark $ is the product of 
two copies of the cone of nilpotent matrices of size $ n $.  
Thus the nilpotent $ K $-orbits in $ \nilpotentvark $ are classified by 
the pairs of partitions: $ \nilpotentvark / K \simeq \partitionsof{n}^2 $.  
The nilpotent orbits in $ \nilpotentvars $ are classified by the signed Young diagrams of size $ 2n $ 
with signature $ (n, n) $ (see \cite{Collingwood.McGovern.1993}, or \cite{Trapa.2005}).  
We denote this set of signed Young diagrams by $ \SYD(n,n) $.  
Thus we have the generalized and exotic Steinberg maps
\begin{equation*}
\Phi^{\theta} : \regTnxTn/S_n \to \partitionsof{n}^2 \quad \text{ and } \quad 
\Phi^{-\theta} : \regTnxTn/S_n \to \SYD(n,n), 
\end{equation*}
both of which are combinatorial.

\begin{theorem}\label{thm:existence.of.algorithm}
There exist efficient combinatorial algorithms which describe the Steinberg maps 
$ \Phi^{\pm \theta} $.
\end{theorem}

Choose $ \omega = \vectwo{\tau_1}{\tau_2} \in \regTnxTn $.  
If $ \tau_1 $ or $ \tau_2 $ is a permutation, then up to the action of $ S_n $, 
we can assume that $ \tau_1 $ or $ \tau_2 $ is equal to $ 1_n $ (the identity matrix).  
In these special cases, the above theorem is already proved in Theorems 9.1 and 10.4 in \cite{Fresse.N.2020}.  
We will explain the algorithms below but 
the full proof of the above theorem will appear elsewhere.

\subsection{Generalizations of the Robinson-Schensted correspondence to type AIII}\label{subsec:gen.RS.correspondence.typeAIII}

To explain the ``efficient combinatorial algorithm'' in Theorem~\ref{thm:existence.of.algorithm}, we need some preparation.  

We choose a partition of $ [n] = \{ 1, \ldots, n \} $ into three disjoint subsets 
$ J, M, M' $ of sizes $ r = \# J ,\;  p = \# M, \; q = \# M' $ so that $ r + p + q = n $.  
Choose another partition into subsets of the same sizes denoted by $ I, L, L' $.  
Let $ \sigma : J \to I $ be a bijection and 
write $ J = \{ j_1 < \cdots < j_r \} $ and $ M = \{ m_1 < \cdots < m_p \} $.  
The same notation applies to the rest of the subsets $ M', I, L, L' $ too.  
Then we put 
\begin{align*}
\tau_1 &= 
\begin{pmatrix}
j_1 & \cdots & j_r & m_1 & \cdots & m_p & m'_1 & \cdots & m'_q \\
\sigma(j_1) & \cdots & \sigma(j_r) & 0 & \cdots & 0 & \ell'_1 & \cdots & \ell'_q
\end{pmatrix} ,
\\
\tau_2 &= 
\begin{pmatrix}
j_1 & \cdots & j_r & m_1 & \cdots & m_p & m'_1 & \cdots & m'_q \\
j_1 & \cdots & j_r & m_1 & \cdots & m_p & 0 & \cdots & 0 
\end{pmatrix} , 
\end{align*}
and consider $ \omega = \vectwo{\tau_1}{\tau_2} \in \regTnxTn $.  
It is straightforward to check the following lemma.

\begin{lemma}\label{lemma:representatives.regTnxTn}
For all the choices of numbers $ r, p, q $, partitions $ [n] = J \sqcup M \sqcup M' = I \sqcup L \sqcup L' $, and bijection
$ \sigma : J \xrightarrow{\;\sim\;} I $ as above, 
$ \{ \omega \} $ is a complete system of representatives for the quotient space $ \regTnxTn / S_n $.  
\end{lemma}

\begin{corollary}
$ \displaystyle \# \dblFV / K = \sum_{n = r + p + q} \hspace*{-.3em} r! \dbinom{n}{r, p, q}^2 , \quad 
\dbinom{n}{r, p, q} = \dfrac{n!}{\,r!\,p!\,q!\,} .
$
\end{corollary}

Let us return to the notation above, and consider $ \omega \in \regTnxTn $.  
Since $ \sigma : J \to I $ is a bijection, 
we can apply the classical Robinson-Schensted(-Knuth) algorithm (see \cite{Fulton}) 
to $ \sigma $ and get a pair of Young tableaux 
$ (T_1, T_2) = (\RSl(\sigma), \RSr(\sigma)) $ of the same shape, 
where $ I $ is the set of entries of $ T_1 $ and $ J $ is that of $ T_2 $.  

There is an algorithm called ``rectification'' of two Young tableaux $ T $ and $ U $ denoted as 
$ \Rect(T \ast U) $ (see \cite{Fulton}).  
We can rectify more than two tableaux at the same time, so we can write $ \Rect(T * U * V) $, for example.  
For a subset $ A \subset [n] $, let 
$ [A] $ be the unique vertical Young tableau of shape $ (1^{\#A}) $ with entries in $A$.  

\begin{theorem}\label{thm:gen.Steinberg.map}
Let $ \omega \in \regTnxTn $ be a representative specified in Lemma~\ref{lemma:representatives.regTnxTn}. 
Put $ \hatT_1 = \Rect([L'] * \!\transpose{\bigl(\RSl(\sigma)\bigr)} * [L]) $ and 
$ \hatT_2 = \Rect([M'] * \!\transpose{\bigl(\RSr(\sigma)\bigr)} * [M]) $, 
where $ \transpose{T} $ denotes the transpose of the tableau $ T $.  
Let $ \lambda = \shape{\hatT_1} $ and $ \mu = \shape{\hatT_2} $.  
Then the image of the generalized Steinberg map is given by 
$ \Phi^{\theta}(\omega) = (\lambda, \mu) $.
\end{theorem}

\begin{remark}
In \cite{Fresse.N.2020}, we gave an algorithm for $ \omega = \vectwo{\tau_1}{1_n} $.  
The formula given here is slightly different from it since one of the Borel subgroups is twisted to the lower triangular one.  
\end{remark}

Note that the above theorem clarifies the fiber of $ \Phi^{\theta} $ and 
gives a different parametrization of $ \dblFV / K $ in terms of standard tableaux.  
We will describe it shortly.  
If a pair of partitions $ \lambda, \nu $ satisfy  
$ \nu_i \leq \lambda_i \leq \nu_i + 1 $ for any $ i \geq 1 $, 
we write $ \nu \vdash \lambda $.  
This is equivalent to saying that $ \nu \subset \lambda $ and 
$ \lambda \setminus \nu $ is a vertical strip.  
Put 
\begin{equation*}
\begin{aligned}
\Upsilon(r,p,q) &= \{ (\lambda, \mu; \lambda', \mu'; \nu) \mid 
\lambda, \mu \in \partitionsof{n}, \; 
\lambda', \mu' \in \partitionsof{r + p}, \; \nu \in \partitionsof{r}, 
\\
& \qquad\qquad \qquad\qquad   \qquad\qquad  
\nu \vdash \lambda' \vdash \lambda, \text{ and } 
\nu \vdash \mu' \vdash \mu \} .
\end{aligned}
\end{equation*}
By the construction, it is easy to see that there exists a bijection 
\begin{equation*}
\genRS : \regTnxTn / S_n \xrightarrow{\;\;\sim\;\;}
\hspace*{-1.5em}
\coprod_{\begin{smallmatrix}
n = r + p + q
\\
(\lambda, \mu; \lambda', \mu'; \nu) \in \Upsilon(r, p, q)
\end{smallmatrix}}
\hspace*{-2em} 
\{ (\hatT_1, \hatT_2; \lambda', \mu'; \nu) \mid 
(\hatT_1, \hatT_2) \in \STab_{\lambda} {\times} \STab_{\mu} \} .
\end{equation*}
We call this bijection $ \genRS $ the \emph{generalized Robinson-Schensted correspondence}.
As before, we get a commutative diagram
\begin{equation*}
\xymatrix  @M 3pt @R 8ex{
\dblFV / K \simeq \regTnxTn /S_n \ar[drr]_{\Phi^{\theta}} \ar@{->}[rr]^-{\sim}_-{\genRS} & & \coprod\limits_{(\star)} 
\{ (\hatT_1, \hatT_2; \lambda', \mu'; \nu) \} \ar[d] & 
\makebox[0ex][c]{\qquad$\ni (\hatT_1, \hatT_2; \lambda', \mu'; \nu)$} \ar@{|->}[d]\\
& & \partitionsof{n}^2 & \makebox[0ex][c]{$\ni (\lambda, \mu) $\qquad} }
\end{equation*}
Here $ (\star) $ represents the condition 
$ n = r + p + q , \; (\lambda, \mu; \lambda', \mu'; \nu) \in \Upsilon(r, p, q) $, 
and $ (\hatT_1, \hatT_2) \in \STab_{\lambda} \times \STab_{\mu} $.

For the formula giving $ \Phi^{-\theta} $, 
we refer the readers to \cite{Fresse.N.2020}.  
In the quoted paper, we only treated the case where 
$ \omega = \vectwo{\tau\;}{1_n} $.  
A complete algorithm will appear elsewhere.

\section{Embedding of the orbits in double flag varieties}\label{sec:abstract.embedding.theory}

In this section, we return back to the situation of 
\S~\ref{S1:Steinberg.theory.for.symmetric.pairs},  
and we do not assume the finiteness of the number of $ K $-orbits in $ \dblFV $ nor 
Assumption~\ref{assumption:exotic.Steinberg.map.to.nilpotents}.

Thus $ (G, K) $ is a symmetric pair and $ P, Q $ are parabolic subgroups 
of $ G $ and $ K $ respectively.  
Sometimes a double flag variety $ \dblFV = G/P \times K/Q $ can be embedded 
into a larger double flag variety.  In this respect, we need a more precise setting.  

Let $ \bbG $ be a connected reductive algebraic group over $ \C $, 
and let $ \theta, \sigma \in \Aut \bbG $ be involutions 
which commute: $ \theta \sigma = \sigma \theta $.  
The fixed point subgroup of $ \theta $ is denoted by $ \bbK = \bbG^{\theta} $ 
so that $ ( \bbG, \bbK ) $ is a symmetric pair.  
Assume that $ G $ is the fixed point subgroup of $ \sigma $, 
$ G = \bbG^{\sigma} $.  
Since $ \theta $ and $ \sigma $ commute, 
$ G $ is stable under $ \theta $ and 
$ \bbK $ is stable under $ \sigma $.
We put 
$ K = G^{\theta} = \bbK^{\sigma} = G \cap \bbK $, 
which is a symmetric subgroup of both $ G $ and $ \bbK $.
\begin{equation*}
\begin{array}{ccc}
\bbK & \subset & \bbG 
\\
\text{\rotatebox[origin=c]{-90}{$\supset$}} & & \text{\rotatebox[origin=c]{-90}{$\supset$}} 
\\
K & \subset & G 
\end{array}
\end{equation*}
\indent
For simplicity, 
we assume that all the subgroups $ G, K $ and $ \bbK $ are connected.

\subsection{Embedding of double flag varieties}

For a parabolic subgroup $ P $ of $ G $, 
we can choose a $ \sigma $-stable parabolic subgroup $ \bbP $ of $ \bbG $ 
which cuts out $ P $ from $ G $, i.e., 
$ P = \bbP^{\sigma} = \bbP \cap G $.  
Similarly, for a parabolic subgroup $ Q $ of $ K $, 
we can choose a $ \sigma $-stable parabolic subgroup\footnote{
In most of the literature, the letter $ \bbQ $ is reserved to denote 
the field of rational numbers.  However, in this paper, we consider everything over 
the complex number field $ \C $ and actually the rational number field does play no rule.  
For this reason, we choose the letter $ \bbQ $ for a parabolic subgroup of $ \bbK $.
}
$ \bbQ $ of $ \bbK $ 
which satisfies 
$ Q = \bbQ^{\sigma} = \bbQ \cap K $.

Denote by $ \bbdblFV = \bbG / \bbP \times \bbK / \bbQ $ a double flag variety 
for $ ( \bbG, \bbK ) $.  
Then $ \dblFV = G/P \times K/Q $ is embedded into $ \bbdblFV $ in a natural manner.  
The embedding 
$ \dblFV \hookrightarrow \bbdblFV $ sends 
$ (g P, k Q) \; (g \in G, k \in K) $ to $ (g \bbP, k \bbQ) $, 
and we do not distinguish $ \dblFV $ from its embedded image in $ \bbdblFV $.  

By abuse of notation, let $ \sigma $ also denote the automorphism of $ \bbdblFV $ defined by 
$ 
\sigma (\bbg \cdot \bbP, \bbk \cdot \bbQ) 
= (\sigma(\bbg) \cdot \bbP, \sigma(\bbk) \cdot \bbQ) 
$ so that 
$ \sigma(h \cdot \bbx) = \sigma(h) \cdot \sigma(\bbx) $ holds 
for $ h \in \bbK $ and $ \bbx \in \bbdblFV $.

For a $ \sigma $-stable subgroup $ \bbH \subset \bbG $, 
we will denote 
\begin{equation*}
\bbH^{-\sigma} = \{ h \in \bbH \mid \sigma(h) = h^{-1} \} .
\end{equation*}

\begin{definition}\label{def:-sigma.square.root}
For a $ \sigma $-stable subgroup $ \bbH $ of $ \bbG $, 
we say that \emph{$ \bbH $ admits $ (-\sigma) $-square roots}  
if for any $ h \in \bbH^{- \sigma} $ there exists an $ f \in \bbH^{-\sigma} $ 
which satisfies $ h = f^2 $.
\end{definition}

\begin{lemma}
If $ \bbP $ and $ \bbQ $ admit $ (-\sigma) $-square roots, then
$ \dblFV $ 
coincides with the fixed point set 
$ \bbdblFV^{\sigma} := \{ \bbx \in \dblFV \mid \sigma(\bbx) = \bbx \} $.  
\end{lemma}

\begin{proof}
The inclusion $ \dblFV \subset \bbdblFV^{\sigma} $ is clear.  
For the other inclusion, 
take $ (\bbg \bbP, \bbk \bbQ)  \in \bbdblFV^{\sigma} $.  
Since $ (\bbg \bbP, \bbk \bbQ)  = \sigma(\bbg \bbP, \bbk \bbQ) 
= (\sigma(\bbg) \bbP, \sigma(\bbk) \bbQ) $, 
we get 
$ \bbg^{-1} \sigma(\bbg) \in \bbP $ and 
$ \bbk^{-1} \sigma(\bbk) \in \bbQ $.  
This means $ \bbg^{-1} \sigma(\bbg) \in \bbP^{-\sigma} $, 
hence by the assumption there exists 
an $ f \in \bbP^{-\sigma} $ such that $ \bbg^{-1} \sigma(\bbg) = f^2 $.  
From this, we get 
$ \bbg f = \sigma(\bbg) f^{-1} = \sigma(\bbg f) $.  
Therefore we conclude that
$ \bbg f \in \bbG^{\sigma} = G $, and 
$ \bbg \bbP = (\bbg f) \bbP $ belongs to $ G/P \subset \bbG / \bbP $.  

Similarly we get 
$ \bbk \bbQ = (\bbk h) \bbQ $ for some $ h \in \bbQ^{-\sigma} $ and 
$ \bbk h \in K $.  
Thus we get 
$ (\bbg \bbP, \bbk \bbQ) = ((\bbg f) \bbP, (\bbk h) \bbQ) \in X $.
\end{proof}

\subsection{Embedding of orbits}\label{subsec:embedding.orbits}

Since the embedding $ \dblFV \hookrightarrow \bbdblFV $ is $ K $-equivariant, 
it induces a natural orbit map defined by 
\footnote{
In the previous sections, we always denote a $ K $-orbit in $ \dblFV $ by 
using a black board bold letter like $ \orbit $.  
However, in this section, we take advantage of systematic notation.
}
\begin{equation*}
\xymatrix @R -6ex @C -6ex @L 5pt
{ \iota : & \dblFV / K \ar[rr] & \qquad \qquad & \bbdblFV / \bbK
\\
& \text{\rotatebox[origin=c]{-90}{$\ni$}} &  & \text{\rotatebox[origin=c]{-90}{$\ni$}} 
\\
& O \ar@{|->}[rr] & & \makebox[1em][l]{\,$ \orbit = \bbK \cdot O $}
} .
\end{equation*}
There is no reason to expect that $ \iota $ is an embedding, i.e., 
to conclude $ O = \orbit \cap \dblFV $.  
In this setting, 
there is a criterion which assures that $ \iota $ is an embedding.  
We quote Assumption 3.1 from \cite{Nishiyama.2014} (with terminology re-adapted):  
``\emph{For any $ x \in \dblFV $, $ \Stab_{\bbK}(x) $ admits $ (-\sigma) $-square roots.}''

To translate the condition in the present setting, 
let $ x = (g \bbP, k \bbQ) =: (\bbP_1, \bbQ_1) $ be a point in $ \dblFV $.  
Note that $ \bbP_1 $ and $ \bbQ_1 $ are $ \sigma $-stable parabolic subgroups 
of $ \bbG $ and $ \bbK $ respectively.  
Then 
$ \Stab_{\bbK}(x) = \bbP_1 \cap \bbQ_1 $.  
Thus the above assumption can be rephrased as 
``\emph{$ \bbP_1 \cap \bbQ_1 $ admits $ (-\sigma) $-square roots.}''

Let us summarize the situation: 
Let $ \bbG, \bbK = \bbG^{\theta}, G = \bbG^{\sigma}, K = \bbK^{\sigma} = G^{\theta} $ be 
as above and assume they are all connected. 
Let $ P $ and $ Q $ be parabolic subgroups of $ G $ and $ K $ respectively and 
choose $ \sigma $-stable parabolic subgroups $ \bbP \subset \bbG $ and 
$ \bbQ \subset \bbK $ such that 
$ P = \bbP \cap G $ and $ Q = \bbQ \cap K $. 
We consider the following two conditions:

\begin{itemize}
\item[(A)] 
$ \bbP $ and $ \bbQ $ admit $ (-\sigma) $-square roots (see Definition~\ref{def:-sigma.square.root}).

\item[(B)]
For any $ \sigma $-stable parabolic subgroups $ \bbP_1 \subset \bbG $ and 
$ \bbQ_1 \subset \bbK $ which are conjugate to $ \bbP $ and $ \bbQ $ respectively, 
the intersection $ \bbP_1 \cap \bbQ_1 $ admits $ (-\sigma) $-square roots.
\end{itemize}

\begin{theorem}\label{thm:embedding.orbits}
In the above setting, 
let us consider the double flag varieties 
$ \dblFV = G/P \times K/Q $ and $ \bbdblFV = \bbG / \bbP \times \bbK / \bbQ $.  
Then the conditions {\upshape (A)} and {\upshape (B)} imply the following.  
\begin{thmenumerate}
\item
There is a natural embedding $ \dblFV \hookrightarrow \bbdblFV $ 
defined by 
$ (g P, k Q) \mapsto (g \bbP, k \bbQ) $  
for any $ g \in G $ and $ k \in K $.
The involutive automorphism 
$ \sigma \in \Aut \bbdblFV $ given by 
$ \sigma(\bbg \bbP, \bbk \bbQ) = (\sigma(\bbg) \bbP, \sigma(\bbk) \bbQ) 
\; (\bbg \in \bbG, \bbk \in \bbK) $ 
is well defined and 
$ \dblFV = \bbdblFV^{\sigma} $ holds.  

\item
The natural orbit map 
$ \iota : \dblFV / K \to \bbdblFV / \bbK $ defined by 
$ \iota(O) = \bbK \cdot O $  for $ O \in \dblFV / K $ is injective, 
i.e., if we put $ \orbit = \iota(O) $, then it holds $ O = \orbit \cap \dblFV = \orbit^{\sigma} $.

\item
For any $ \bbK $-orbit $ \orbit $ in $ \bbdblFV $, 
the intersection $ \orbit \cap \dblFV $ is either empty or a single $ K $-orbit.
\end{thmenumerate}
\end{theorem}

As we have already explained above, this theorem follows from \cite[Theorem 3.2]{Nishiyama.2014}.  
We make emphasis on the fact that we do not need finiteness of the orbits.  

If we take $ Q = K $ and $ \bbQ = \bbK $, 
we get the following corollary, 
which is known to experts on the basis of case-by-case analysis.

\begin{corollary}\label{cor:embedding.K.orbits.in.flag.variety}
Let $ \bbP $ be a $ \sigma $-stable parabolic subgroup of $ \bbG $, which admits $ (-\sigma) $-square roots.  
Assume that, for any $ \sigma $-stable parabolic subgroup $ \bbP_1 \subset \bbG $ conjugate to $ \bbP $, 
the intersection $ \bbP_1 \cap \bbK $ admits $ (-\sigma) $-square roots.
Then the natural embedding map 
$ G/P \to \bbG/\bbP $ induces an embedding of orbits 
$ K \backslash G/P \to \bbK \backslash \bbG/\bbP $, i.e., 
for any $ \bbK $-orbit $ \orbit $ in $ \bbG/\bbP $, 
the intersection $ \orbit \cap G/P $ is either empty or a single $ K $-orbit.
\end{corollary}

In this way, the classification of the $ K $-orbits in 
the partial flag variety $ G/P $ reduces to that of $ \bbK $-orbits in $ \bbG/\bbP $ 
together with the determination of the subset of orbits with non-empty intersection with $ G/P $.  
Note that, in this case, both $ K \backslash G/P $ and $ \bbK \backslash \bbG/\bbP $ 
are finite sets.

\subsection{Embedding of conormal bundles}

We fix a nondegenerate invariant bilinear form on 
$ \lie{G} := \Lie(\bbG) $ which is also invariant under involutions $ \sigma $ and $ \theta $.  
Certainly it always exists and by restrictions it descends to 
nondegenerate bilinear forms on $ \lie{g}, \lie{k} $ and $ \lie{K} := \Lie(\bbK) $.  
By this invariant form, we identify $ \lie{G} $ and $ \lie{G}^{\ast} $, 
and 
similar identifications take place for $ \lie{g}, \lie{k} $ and $ \lie{K} $.  
In the following, we will use German capital letters for the Lie algebras of 
the groups denoted by black board bold letters.

Let us consider the cotangent bundles 
$ T^*\dblFV $ and $ T^*\bbdblFV $.  
Since 
$ T^*\bbdblFV 
\simeq (\bbG \times_{\bbP} \nilradical{\lie{P}}) 
\times (\bbK \times_{\bbQ} \nilradical{\lie{Q}}) $ 
by the identification $ \lie{G} \simeq \lie{G}^{\ast} $ and $ \lie{K} \simeq \lie{K}^{\ast} $, 
we have a natural extension of $ \sigma \in \Aut \bbdblFV $ to the whole $ T^*\bbdblFV $.  
We prove

\begin{proposition}
Let us assume the condition {\upshape(A)} in {\upshape\S~\ref{subsec:embedding.orbits}}. 
Then $ (T^*\bbdblFV)^{\sigma} = T^*\dblFV $ holds.
\end{proposition}

\begin{proof}
It is obvious that $ T^*\dblFV \subset (T^*\bbdblFV)^{\sigma} $.  
Let us prove the reversed inclusion.  
Take $ [\bbg, \bbu] \in \bbG \times_{\bbP} \nilradical{\lie{P}} $.  
Then $ \sigma([\bbg, \bbu]) = [\sigma(\bbg), \sigma(\bbu)] $ is equal to $ [\bbg, \bbu] $ 
if and only if there exists $ \bbp \in \bbP $ such that 
$ \sigma(\bbg) = \bbg \bbp $ and $ \sigma(\bbu) = {\bbp}^{-1} {\bbu} $.
From the first equality, we get $ \bbp = \bbg^{-1} \sigma(\bbg) \in \bbP^{-\sigma} $ 
so that there exists a square root $ f \in \bbP^{-\sigma} $ of $ \bbg^{-1} \sigma(\bbg) $.  
Then $ [\bbg, \bbu] = [\bbg f, f^{-1} \bbu] \in G \times_{P} \nilradical{\lie{p}} $.
In fact, 
$ \sigma(\bbg f) = \sigma(\bbg) f^{-1} = \bbg f $ which implies $ \bbg f \in G $.  
Similarly, 
$ \sigma(f^{-1} \bbu) = f \sigma(\bbu) = f \bbp^{-1} \bbu = f f^{-2} \bbu = f^{-1} \bbu $, 
which proves $ f^{-1} \bbu \in (\nilradical{\lie{P}})^{\sigma} = \nilradical{\lie{p}} $.  

In the same manner, we conclude $ (\bbK \times_{\bbQ} \nilradical{\lie{Q}})^{\sigma} 
= K \times_Q \nilradical{\lie{q}} $, which proves the proposition.
\end{proof}

Let us denote the moment maps by 
\begin{equation}
\mu_{\bbdblFV} : T^*\bbdblFV \to \lie{K}^{\ast} \simeq \lie{K} , 
\qquad
\mu_{\dblFV} : T^*\dblFV \to \lie{k}^{\ast} \simeq \lie{k} .
\end{equation}
Clearly $ \mu_{\bbdblFV} $ commutes with $ \sigma $ and 
$ \mu_{\bbdblFV} \restrict_{T^*\dblFV} = \mu_{\dblFV} $.  
Let $ \conormalvariety_{\bbdblFV} = \mu_{\bbdblFV}^{-1}(0) $ 
and $ \conormalvariety_{\dblFV} = \mu_{\dblFV}^{-1}(0) $ 
be the corresponding conormal varieties
\footnote{
We defined the conormal variety in \S~\ref{S1:Steinberg.theory.for.symmetric.pairs} in a different way, but they coincide.  
See Fact~\ref{fact:conormal.var.is.null.fiber}{}\eqref{fact:conormal.var.is.null.fiber:item:1} and also \cite{Fresse.N.2020}.}.

\begin{theorem}\label{thm:embedding.conormal.bundles}
Assume both conditions 
{\upshape(A)} and {\upshape(B)} in {\upshape\S~\ref{subsec:embedding.orbits}}.
\begin{thmenumerate}
\item\label{thm:embedding.conormal.bundles:item:01}
It holds that $ (\conormalvariety_{\bbdblFV})^{\sigma} = \conormalvariety_{\dblFV} $.  
\item\label{thm:embedding.conormal.bundles:item:02}
By Theorem~\ref{thm:embedding.orbits}, 
the orbit map $ \iota : \dblFV / K \to \bbdblFV / \bbK $ is injective.  
For $ O \in X/K $, put $ \orbit = \bbK \cdot O = \iota(O) $.  
Then we have $ (T^*_{\orbit}\bbdblFV)^{\sigma} = T^*_O\dblFV $.
\end{thmenumerate}
\end{theorem}

\begin{proof}
\eqref{thm:embedding.conormal.bundles:item:01}\ 
$ (\conormalvariety_{\bbdblFV})^{\sigma} 
= (T^*\bbdblFV)^{\sigma} \cap \mu_{\bbdblFV}^{-1}(0)
= T^*\dblFV \cap \mu_{\bbdblFV}^{-1}(0)
= T^*\dblFV \cap \mu_{\dblFV}^{-1}(0)
= \conormalvariety_{\dblFV} $.  

\eqref{thm:embedding.conormal.bundles:item:02}\ 
Since $ \orbit^{\sigma} = O $, it suffices to check the equality for the fiber.  
Take $ x \in O $ and write $ x = (P_1, Q_1) = (\bbP_1, \bbQ_1) $ 
(as a point in $ G/P \times K/Q $ embedded into $ \bbG/\bbP \times \bbK/\bbQ $).  
Then the fiber of the conormal bundle at $ x $ is given by 
$ (T^*_{\orbit}\bbdblFV)_x = \bigl( (\lie{P}_1, \lie{Q}_1) + \Delta \lie{K} \bigr)^{\bot} $, 
where $ \Delta \lie{K} $ denotes the diagonal embedding and $ \bot $ refers to the orthogonal 
in the dual space.  
Since all the subspaces which are relevant are $ \sigma $-stable, we conclude 
$ (T^*_{\orbit}\bbdblFV)_x^{\sigma} 
= \bigl( (\lie{P}_1^{\sigma}, \lie{Q}_1^{\sigma}) + \Delta \lie{K}^{\sigma}\bigr)^{\bot} 
= \bigl( (\lie{p}_1, \lie{q}_1) + \Delta \lie{k}\bigr)^{\bot} 
= (T^*_{O}\dblFV)_x
$.  
\end{proof}

\subsection{Compatibility with Steinberg maps}
\label{section:compatibility.Steinberg.maps}
In this section we assume both conditions (A) and (B), so that Theorem \ref{thm:embedding.orbits} provides us with an orbit embedding $\iota:\dblFV / K \hookrightarrow \bbdblFV / \bbK $.

Recall that there are two types of nilpotent cones $ \nilpotents^{\pm\theta} $ in 
$ \lie{k} = \lie{g}^{\theta} $ and $ \lie{s} = \lie{g}^{-\theta} $ respectively.  
The symmetric subgroup $ K $ acts on both nilpotent varieties with finitely many orbits.  

Similarly, we denote the nilpotent variety for $ \lie{G} = \Lie (\bbG) $ 
by $ \Nilpotents = \Nilpotents_{\lie{G}} $ and put 
$ \Nilpotentvark = \Nilpotents \cap \lie{K} $ and 
$ \Nilpotentvars = \Nilpotents \cap \lie{G}^{-\theta} $.  
The symmetric subgroup $ \bbK $ acts on $ \Nilpotents^{\pm\theta} $ with finitely many orbits.

Clearly $ \nilpotents^{\pm\theta} $ are fixed point varieties of 
$ \Nilpotents^{\pm\theta} $ by the involution $ \sigma $.  
Therefore we have natural orbit maps 
$ \iota^{\pm\theta} : \nilpotents^{\pm\theta} / K \to \Nilpotents^{\pm\theta} / \bbK $.  
Namely, for an orbit $ \calorbit \in \nilpotents^{\pm\theta} / K $, 
the respective map is defined by letting
$ \iota^{\pm\theta}(\calorbit) = \bbK \cdot \calorbit $.
The maps $ \iota^{\pm\theta} $ are not injective in general.

Assumption \ref{assumption:exotic.Steinberg.map.to.nilpotents} clearly holds for $\dblFV$ whenever it holds for $\bbdblFV$, and in this case the construction of Section
\ref{S1:Steinberg.theory.for.symmetric.pairs} yields the Steinberg maps $\Phi^{\pm\theta}:\dblFV/K\to \nilpotents^{\pm\theta}/K$ relative to $\dblFV$
and $\bbPhi^{\pm\theta}:\bbdblFV/\bbK \to \Nilpotents^{\pm\theta}/K$ relative to $\bbdblFV$. 
Then it is natural to make the following

\begin{conjecture}\label{conj:question.compatibility}
The diagram below commutes.  
In other words, the orbit embedding $\iota:X/K\hookrightarrow \bbX/\bbK$ is compatible with the Steinberg maps $\Phi^{\pm\theta}$ and $\bbPhi^{\pm\theta}$.
\begin{equation}\label{eq:embedding.and.Steinberg.maps.commute}
\vcenter{
\xymatrix 
{ \dblFV / K \ar@{^{(}->}[d]^\iota \ar[rr]^{\text{\normalsize$\Phi^{\pm\theta}$}} & \qquad \qquad & \nilpotents^{\pm\theta} / K \ar[d]^{\iota^{\pm\theta}} \\
 \bbdblFV / \bbK \ar[rr]^{\text{\normalsize$\bbPhi^{\pm\theta}$}} & \qquad \qquad & \Nilpotents^{\pm\theta} / \bbK} 
}
\end{equation}
\end{conjecture}

The functoriality of the constructions is not sufficient to hold 
Conjecture~\ref{conj:question.compatibility} affirmatively. Let $O\subset \dblFV$ be a $K$-orbit
and let $\orbit:=\iota(O)$ be the corresponding $\bbK$-orbit in $\bbdblFV$.
By definition of the maps $\bbPhi^{\pm\theta}$, there is a dense open subset
$\bbU$ of the conormal bundle $T^*_\orbit\bbdblFV$ such that
\[
\bbu\in\bbU
\quad
\Longrightarrow
\quad
(\bbu^\theta\in \bbPhi^\theta(\orbit)\quad\mbox{and}\quad \bbu^{-\theta}\in \bbPhi^{-\theta}(\orbit)).
\]
We may assume that $\bbU \subset T^*_\orbit\bbdblFV$ is $\bbK$-stable and $\sigma$-stable.  
Note that $T^*_O\dblFV=(T^*_\orbit\bbdblFV)^\sigma$ by Theorem \ref{thm:embedding.conormal.bundles}.

\begin{lemma}\label{lemma:dense.U.proves.commutativity}
The equality 
$\bbPhi^{\pm\theta}(\iota(O))=\iota^{\pm\theta}(\Phi^{\pm\theta}(O))$ 
holds (and hence the diagram \eqref{eq:embedding.and.Steinberg.maps.commute} commutes) if $\bbU$ intersects $T^*_O\dblFV$.
\end{lemma}

However, in general, it seems difficult to check the condition for $ \bbU $ in the above lemma.  
In the next section, we will discuss an explicit embedding of type CI into type AIII 
(see \eqref{eq:bbK.G.K.typeCI.embedding} below).  
Yet it is still difficult to prove the commutativity of the Steinberg maps in full generality.  
Direct calculations tell us that up to $ n = 3 $ there always exist such $ \bbU $ in 
Lemma \ref{lemma:dense.U.proves.commutativity}, so that the commutativity holds.

\section{Steinberg theory for type CI embedded into type AIII}\label{section:embedding.typeCI.into.typeAIII}

In this section, we apply the theory of embedding to the case of type CI inside type AIII.  
This section 
serves as a meaningful example of the embedding theory given in 
\S~\ref{sec:abstract.embedding.theory}.

Let us begin with 
$ \bbG = \GL_{2n} $ and 
two commuting involutions $ \sigma, \theta \in \Aut \bbG $ defined by 
\begin{align*}
\theta(g) = I_{n,n}^{-1}\, g\, I_{n,n} & & I_{n,n} = \diag (1_n, -1_n) ,
\\
\sigma(g) = J_n^{-1}\,\transpose{g}^{-1} J_n & & J_n = \mattwo{}{-1_n}{1_n}{} .
\end{align*}
Then 
\begin{equation}\label{eq:bbK.G.K.typeCI.embedding}
\bbK = \bbG^{\theta} = \GL_n \times \GL_n , \quad 
G = \bbG^{\sigma} = \Sp_{2n} , \quad
K = G^{\theta} \simeq \GL_n
\end{equation}
are all connected.  
We use $ J_n $ to define a symplectic form on $ V = \C^{2n} $ by 
$ (u, v) = \transpose{u} J_n v $.  
Take $ Q = B = B_n^+ $ in $ K = \GL_n $ as a Borel subgroup and let
$ P = P_S $ be the Siegel parabolic subgroup of $ G = \Sp_{2n} $ 
which stabilizes the Lagrangian subspace 
$ V^+ = \langle \eb_1, \eb_2, \dots, \eb_n \rangle \subset V $, 
where $ \eb_k $ denotes the $ k $-th elementary basis vector.  
Similarly we put $ V^- = \langle \eb_{n + 1}, \eb_{n + 2}, \dots, \eb_{2 n} \rangle $, 
which is also Lagrangian, and 
$ V = V^+ \oplus V^- $ is a polarization stable under $ \bbK $. 

We take 
$ \bbBK = B_n^+ \times B_n^- $ as a $ \sigma $-stable Borel subgroup of $ \bbK $ and 
put $ \bbP = \Stab_{\bbG}(V^+) $ the stabilizer of $ V^+ $ in 
$ \bbG = \GL_{2n} $.  It is easy to see that $ \bbP $ is a $ \sigma $-stable parabolic subgroup.
Thus our double flag varieties are 
\begin{equation}\label{eq:dblflag.Sp2n.GLn}
\dblFV = G/P \times K/B = \Sp_{2n}/P_S \times \GL_n/B_n^+ \simeq \LGrass(\C^{2n}) \times \FlagVar_n , 
\end{equation}
and 
\begin{equation}\label{eq:dblflag.GL2n.GLnGLn}
\begin{aligned}[t]
\bbdblFV &= \bbG/\bbP \times \bbK/\bbBK = \GL_{2n}/P_{(n,n)} \times \GL_n/B_n^+ \times \GL_n/B_n^- 
\\
&\simeq \Grass_n(\C^{2n}) \times \FlagVar_n \times \FlagVar_n ,
\end{aligned}
\end{equation}
where $ \LGrass(V) $ denotes the Lagrangian Grassmannian, the variety of all the Lagrangian subspaces 
in the symplectic vector space $ V = \C^{2n} $ with the symplectic form defined by $ J_n $, 
and $ \FlagVar_n $ is the variety of complete flags in $ \C^n $.  
We use the notation $ P_{(n,n)} $ for the parabolic subgroup in $ \GL_{2n} $ determined by 
the partition $ (n, n) $ of $ 2n $.

\subsection{Assumptions (A) and (B)}

In this setting, 
Assumptions (A) and (B) in \S~\ref{subsec:embedding.orbits} are satisfied as we will see below.  
So the whole theory of embedding works well.  
Note that in this case 
$ \bbdblFV $ has finitely many $ \bbK $-orbits (see \cite{Fresse.N.2020}), 
hence we conclude that $ \dblFV $ has also finitely many $ K $-orbits, 
although the finiteness of orbits is already known by a general classification theory 
\cite{HNOO.2013}.

First, let us check the condition (A).

\begin{lemma}\label{lemma:typeC.condition(A)}
In the setting above, 
$ \bbBK \subset \bbK $ and $ \bbP \subset \bbG $ admit 
$ (-\sigma) $-square roots.
\end{lemma}

\begin{proof}
We follow the strategy of Ohta \cite{Ohta.2008} (the proof of Theorem 1 there).

Take $ g \in \bbP^{-\sigma} $.  
Then $ \sigma(g) = J_n^{-1} \transpose{g}^{-1} J_n = g^{-1} $ and 
thus we get $ J_n^{-1} \transpose{g} J_n = g $.  
The expression of the left-hand side can be extended into a mapping on the whole matrix algebra $ \Mat_{2n} $, 
which we denote by $ \tau $ temporarily. 
So we put 
$ \tau(A) = J_n^{-1} \transpose{A} J_n $ for $ A \in \Mat_{2n} $, 
which is an anti-automorphism of $ \Mat_{2n} $.  

For the above chosen $ g \in \bbP^{-\sigma} $, 
a standard argument of linear algebra produces a polynomial 
$ f(T) \in \C[T] $ such that $ f(g)^2 = g $.
Since $ g $ is invertible, $ f(g) $ is also invertible and 
$ f(g) \in \GL_{2n} = \bbG $.  
Since $ \tau(f(g)) = f(\tau(g)) = f(g) $, 
this means $ f(g) \in \bbG^{-\sigma} $.  

Note that $ g \in \bbP $ stabilizes $ V^+ $.  
This forces $ f(g) $ to stabilize $ V^+ $ also, hence $ f(g) \in \bbP^{-\sigma} $.  
This proves that $ \bbP $ admits $ (-\sigma) $-square roots.

Using the fact that $ \bbBK $ is the stabilizer in $ \bbK $ of a complete flag in $ V^{\pm} $, 
we can prove that $ \bbBK $ admits $ (-\sigma) $-square roots in the same way.  
However, note that $ \bbK $ is the stabilizer of the polar decomposition 
$ V^+ \oplus V^- $.
\end{proof}

\begin{remark}
The same argument proves that \emph{any} $ \sigma $-stable parabolic subgroup in $ \bbG $ or  
$ \bbK $ admits $ (-\sigma) $-square roots (cf.~Corollary~\ref{cor:embedding.K.orbits.in.flag.variety}).
\end{remark}

To prove the condition (B), 
we will consider $ \bbP_1 \cap (\bbBK)_1 $ for an arbitrary $ \sigma $-stable 
parabolic subgroup $ \bbP_1 \subset \bbG $ and a Borel subgroup $ (\bbBK)_1 \subset \bbK $.  
An element $ g \in \bbP_1 \cap (\bbBK)_1 $ can be characterized by the property 
that it stabilizes various subspaces (and also polar decomposition) of $ V $.  
Thus the literary same arguments in Lemma~\ref{lemma:typeC.condition(A)} can apply, which proves (B).

The conditions (A) and (B) imply the following theorem.

\begin{theorem}\label{thm:abstract.embedding.typeCI.to.typeAIII}
Let $ \dblFV $ and $ \bbdblFV $ be the double flag varieties defined in 
\eqref{eq:dblflag.Sp2n.GLn} and 
\eqref{eq:dblflag.GL2n.GLnGLn} respectively.  
Then the orbit map $ \dblFV / K \to \bbdblFV / \bbK $ is injective, i.e., 
for any $ \bbK $-orbit $ \orbit $ in $ \bbdblFV $, 
the intersection $ X \cap \orbit $ is either empty or a single $ K $-orbit.
\end{theorem}

\subsection{Explicit embedding of orbits in double flag varieties}\label{subsec:explicit.embedding.orbits.in.dblFV.typeCI}

Since we get an abstract embedding theorem for orbits, 
we can use it to classify the orbits in the double flag variety $ \dblFV $ explicitly 
in terms of the parametrization of the orbits in $ \bbdblFV $ 
given in Theorem~\ref{thm:orbits.in.dblFV.typeAIII}.  

First, we clarify the explicit embedding map $ \dblFV \to \bbdblFV $ and the involutive automorphism on $ \bbdblFV $.  
Recall the various identifications (see \S~\ref{subsec:symmetric.pair.typeAIII}):
\begin{equation}\label{eqn:bbdblFV.bbK.various.forms.typeAIII}
\begin{aligned}
\bbdblFV / \bbK &\simeq B_n^+ \times B_n^- \backslash \GL_{2n} / P_{(n,n)} 
\simeq B_n^+ \times B_n^- \backslash \Grass_n(\C^{2n}) 
\\
&\simeq B_n^+ \times B_n^- \backslash \regMat_{2n,n} / \GL_n 
\simeq \regTnxTn / S_n ,
\end{aligned}
\end{equation}
where $ \regMat_{2n,n} $ denotes the set of $ 2n $ by $ n $ matrices of full rank (i.e., rank $ n $).  
Take 
$ \omega = \vectwo{\tau_1}{\tau_2} \in \regTnxTn $ as a representative of 
a $ \bbK $-orbit $ \orbit $, 
and denote $ \orbit = \orbit_{\omega} $.  
Then $ \orbit_{\omega} $ is generated by the point $ (\bbBK, \bbg \bbP) $, 
where 
\begin{equation*}
\bbg = \mattwo{\tau_1}{\xi_1}{\tau_2}{-\xi_2} \in \GL_{2n} = \bbG 
\end{equation*}
for some $ \xi_1, \xi_2 \in \Mat_n $ (we put minus sign in front of $ \xi_2 $ for later convenience).  
In fact, it is easy to see
\begin{equation*}
\bbg \bbP \cdot V^+ = \bbg V^+ = [\omega], 
\end{equation*}
where $ [\omega] $ denotes the image of the matrix $ \omega $, i.e.,
the $ n $-dimensional subspace generated by the column vectors of $ \omega $.  
We need a lemma.

\begin{lemma}\label{lemma:good.choice.of.g.xi1.xi2}
We can choose $ \bbg = \mattwo{\tau_1}{\xi_1}{\tau_2}{-\xi_2} $ in such a way that 
$ \vectwo{\xi_1}{\xi_2} \in \regTnxTn  $ and 
$ \transpose{\xi_1} \tau_1 - \transpose{\xi_2} \tau_2 = 0 $.
\end{lemma}

\begin{proof}
Based on the general description of $ \tau_1 $ and $ \tau_2 $ given in Section \ref{subsec:gen.RS.correspondence.typeAIII}, we can find permutation matrices
$ s_1,s_2,s\in S_n $ such that
\begin{equation*}
s_1\tau_1 s = \matthree
{1_r}{0}{0}
{0}{0}{0} 
{0}{0}{1_q}
\;\;
\quad
\text{ and }
\quad
s_2\tau_2 s = \matthree
{1_r}{0}{0}
{0}{1_p}{0}
{0}{0}{0} .
\end{equation*}
Then, choosing
$ \xi_1 $ and $ \xi_2 $ so that
\begin{equation*}
s_1\xi_1 = \matthree
{1_r}{0}{0}
{0}{1_p}{0} 
{0}{0}{0} 
\quad 
\text{ and }
\quad
s_2\xi_2 = \matthree
{1_r}{0}{0}
{0}{0}{0}
{0}{0}{1_q} ,
\end{equation*}
we get 
\begin{equation*}
0 
= \transpose{(s_1\xi_1)} (s_1 \tau_1 s) - \transpose{(s_2\xi_2)} (s_2 \tau_2 s) 
= (\transpose{\xi_1} \tau_1 - \transpose{\xi_2} \tau_2) s , 
\end{equation*}
whence $ \transpose{\xi_1}\tau_1-\transpose{\xi_2}\tau_2=0 $ as desired.
\end{proof}

We define an involution on $ \regTnxTn / S_n \simeq \bbdblFV / \bbK $ by 
$ \sigma(\orbit_{\omega}) = \orbit_{\sigma(\omega)} \; (\omega \in \regTnxTn) $, 
which is denoted by the same letter $ \sigma $ by abuse of notation.  
Note that $ \sigma(\omega) $ is determined only modulo the right multiplication by $ S_n $.  

\begin{proposition}\label{prop:explicit.sigma.on.bbdblFV.bbK}
For $ \bbg = \mattwo{\tau_1}{\xi_1}{\tau_2}{-\xi_2} $ chosen in 
Lemma~\ref{lemma:good.choice.of.g.xi1.xi2} and $ \omega = \vectwo{\tau_1}{\tau_2} \in \regTnxTn $, 
let $ \orbit_{\omega} $ be the $ \bbK $-orbit in $ \bbdblFV $ through the point 
$ (\bbBK, \bbg \bbP) $.  
Then $ \sigma(\orbit_{\omega}) = \orbit_{\sigma(\omega)} $ is given by 
\begin{equation*}
\sigma(\omega) = \sigma\Bigl( \vectwo{\tau_1}{\tau_2} \Bigr)
= \vectwo{\xi_2}{\xi_1} .
\end{equation*}
\end{proposition}

\begin{proof}
We have $ \sigma\bigl( (\bbBK, \bbg \bbP) \bigr) = (\bbBK, \sigma(\bbg) \bbP) $ and 
$ \sigma(\bbg) = J_n^{-1} \transpose{\bbg}^{-1} J_n $. 
Let us compute $ \transpose{\bbg}^{-1} $.  We get 
\begin{equation*}
\transpose{\bbg} \bbg 
= \mattwo{\transpose{\tau_1}}{\transpose{\tau_2}}{\transpose{\xi_1}}{-\transpose{\xi_2}} 
  \mattwo{\tau_1}{\xi_1}{\tau_2}{-\xi_2} 
= \mattwo{\transpose{\tau_1} \tau_1 + \transpose{\tau_2} \tau_2}{\transpose{\tau_1} \xi_1 - \transpose{\tau_2} \xi_2}{\transpose{\xi_1} \tau_1 - \transpose{\xi_2} \tau_2}{\transpose{\xi_1} \xi_1 + \transpose{\xi_2} \xi_2}
=: \mattwo{d_1}{0}{0}{d_2} , 
\end{equation*}
by the property $ \transpose{\xi_1} \tau_1 - \transpose{\xi_2} \tau_2 = 0 $.  
An easy calculation tells that $ d_1 = \transpose{\tau_1} \tau_1 + \transpose{\tau_2} \tau_2 $ 
is a diagonal matrix with diagonal entries $ 1 $ or $ 2 $, 
and so is $ d_2 = \transpose{\xi_1} \xi_1 + \transpose{\xi_2} \xi_2 $.  
Thus we get 
$ \transpose{\bbg}^{-1} = \bbg \bbd^{-1} $ with $ \bbd = \diag (d_1, d_2) $.  
From this, we compute 
\begin{equation*}
\sigma(\bbg) = J_n^{-1} \transpose{\bbg}^{-1} J_n 
= J_n^{-1} \bbg \bbd^{-1} J_n 
= \mattwo{\xi_2}{\tau_2}{\xi_1}{-\tau_1}
\cdot \diag(-d_2^{-1}, -d_1^{-1}) , 
\end{equation*}
and get $ \sigma(\bbg) \bbP = \mattwo{\xi_2}{\tau_2}{\xi_1}{-\tau_1} \bbP $.  
This implies $ \sigma\Bigl( \vectwo{\tau_1}{\tau_2} \Bigr)
= \vectwo{\xi_2}{\xi_1} $.
\end{proof}

\begin{theorem}\label{thm:explicit.embedding.typeCI.to.typeAIII}
For $ \omega = \vectwo{\tau_1}{\tau_2} \in \regTnxTn / S_n $, 
let $ \orbit_{\omega} $ be the corresponding $ \bbK $-orbit in $ \bbdblFV $.  
Then the following {\upshape(1)--(4)} are all equivalent.
\begin{thmenumerate}
\item
$ \orbit_{\omega} \cap X \neq \emptyset $ 
(and consequently it is a single $ K $-orbit); 
\item
$ \sigma( \orbit_{\omega} ) = \orbit_{\omega} $, i.e., the $ \bbK $-orbit is $ \sigma $-stable;
\item
$ \transpose{\tau_1} \tau_2 \in \Sym_n $;
\item
$ \transpose{\tau_2} \tau_1 \in \Sym_n $.
\end{thmenumerate}
In particular, the set of $ K $-orbits in the double flag variety 
$ \dblFV = \Sp_{2n} / P_S \times \GL_n / B_n^+ $ 
of type CI is parametrized by $ \regCnxCn / S_n $, where 
\begin{equation*}
\regCnxCn := 
\bigl\{ \omega = \vectwo{\tau_1}{\tau_2} \in \regTnxTn \bigm| 
\transpose{\tau_1} \tau_2 = \transpose{\tau_2} \tau_1 \in \Sym_n \bigr\} .
\end{equation*}
\end{theorem}

As in the case of $ \bbdblFV / \bbK $ (see \eqref{eqn:bbdblFV.bbK.various.forms.typeAIII}), 
there are natural bijections
\begin{equation}\label{eqn:dblFV.K.various.forms.typeCI}
\begin{aligned}
\dblFV / K &\simeq (\LGrass(\C^{2n}) \times \FlagVar_n) / \GL_n 
\simeq B_n^+ \backslash \Sp_{2n} / P_S
\\
&\simeq B_n^+ \backslash \LGrass(\C^{2n}) 
\simeq B_n^+ \backslash \regCMat_{2n,n} / \GL_n ,
\end{aligned}
\end{equation}
where 
\begin{equation}
\begin{aligned}
\regCMat_{2n,n} 
&= \{ A \in \regMat_{2n,n} \mid \transpose{A} J_n A = 0 \} 
\\
&= \{ A = \vectwo{A_1}{A_2} \in \regMat_{2n,n} \mid \transpose{A_1} A_2 = \transpose{A_2} A_1 
\in \Sym_n \} . 
\end{aligned}
\end{equation}
Note that the actions of $ b \in B_n^+ $ are all defined by the left multiplications by 
$ \mattwo{b}{0}{0}{\transpose{b}^{-1}} $.  
The above theorem tells that all these coset spaces are in bijection with 
$ \regCnxCn / S_n $.  

We get an interesting corollary as a byproduct.

\begin{corollary}
For $ \omega = \vectwo{\tau_1}{\tau_2} \in \regTnxTn $, 
the following are all equivalent.
\begin{thmenumerate}
\item
$ \omega \in \regCnxCn $, i.e., 
$ \transpose{\tau_1} \tau_2 = \transpose{\tau_2} \tau_1 \in \Sym_n $ holds;
\item
$ \transpose{\tau_1} b \tau_2 \in \Sym_n $ for some $ b \in B_n^- $;
\item
$ \transpose{\tau_1} b' \tau_2 \in \Sym_n $ for some $ s \in S_n $ 
and $ b' \in s B_n^- s^{-1} $.
\item
$ \transpose{\tau_1} b' \tau_2 \in \Sym_n $ for any $ s \in S_n $ 
and some $ b' \in s B_n^- s^{-1} $.
\end{thmenumerate}
\end{corollary}

\begin{proof}
Obviously $ \vectwo{\tau_1}{\tau_2} \in \regCnxCn $ if and only if 
$ \vectwo{s \tau_1}{s \tau_2} \in \regCnxCn $ for some $ s \in S_n $
if and only if it is so for any $ s \in S_n $.
Therefore it suffices to prove the equivalence of (1) and (2).
Since the implication $ (1) \implies (2) $ is clear, let us prove that (2) implies (1).  
Note that $ \omega = \vectwo{\tau_1}{\tau_2} $ and 
$ \vectwo{\transpose{b} \tau_1}{\tau_2} = \diag(\transpose{b}, 1) \cdot \omega $ 
are in the same orbit $ \orbit_{\omega} $.  
Thus we get 
\begin{equation*}
(2) \iff 
\Bigl[ \vectwo{\transpose{b} \tau_1}{\tau_2} \Bigr] \text{ is Lagrangian} 
\implies \orbit_{\omega} \cap X \neq \emptyset 
\iff (1) .
\end{equation*}
\end{proof}

\subsection{Embedding of nilpotent orbits}\label{subsec:embedding.nilpotent.orbits.typeCI}

The embedding of nilpotent orbits in type CI into those of type AIII is well studied.  
Let us make a quick summary of the known facts.

We use the general notation for nilpotent cones introduced in \S~\ref{section:compatibility.Steinberg.maps}.
In particular we consider the nilpotent cones $ \nilpotents^{\pm\theta}\subset\lie{g}^{\pm\theta} $ 
of the Lie algebra $\lie{g}=\Lie(\Sp_{2n})$.
Note that $\lie{k}:=\lie{g}^\theta$ identifies with $\Lie(\GL_n)$ embedded inside $\lie{g}$.
The symmetric subgroup $K=\GL_n$ acts on $ \nilpotents^{\pm\theta} $ with finitely many orbits.

Similarly, the two nilpotent cones for $ \lie{G} = \Lie (\GL_{2n}) $ 
are denoted by 
$ \Nilpotentvark = \Nilpotents \cap \lie{K} $ and 
$ \Nilpotentvars = \Nilpotents \cap \lie{G}^{-\theta} $.  
The symmetric subgroup $ \bbK = \GL_n \times \GL_n $ acts on $ \Nilpotents^{\pm\theta} $.

Finally recall the natural orbit maps 
$ \iota^{\pm\theta} : \nilpotents^{\pm\theta} / K \to \Nilpotents^{\pm\theta} / \bbK $,
$\calorbit\mapsto \bbK \cdot \calorbit $.

\begin{theorem}[Ohta \cite{Ohta.1986,Ohta.1991.TMJ}]
\label{thm:ohta}
In the above setting, 
the orbit maps $ \iota^{\pm\theta} $ are injective and respect the closure ordering.  
In particular, for any nilpotent $ \bbK $-orbit $ \Orbit \in \Nilpotents^{\pm\theta} / \bbK $, 
the intersection 
$ \Orbit \cap \nilpotents^{\pm\theta} $ is either empty or a single nilpotent $ K $-orbit.  
Moreover, if $ \Orbit $ is $ \sigma $-stable, then 
$ \Orbit \cap \nilpotents^{\pm\theta} $ is nonempty.
\end{theorem}

Note that the above theorem \emph{does not} imply that
$ \closure{\Orbit} \cap \nilpotents^{\pm\theta} $ is irreducible.  
However, if $ \Orbit \cap \nilpotents^{\pm\theta} \neq \emptyset $, 
then $ \closure{\Orbit} \cap \nilpotents^{\pm\theta} 
= \closure{\Orbit \cap \nilpotents^{\pm\theta}} $ is irreducible, and 
in fact it is the closure of a nilpotent $ K $-orbit.  

Let us give more explicit information on the embedding of nilpotent orbits together with parametrization.  

Since $ \Nilpotentvark $ is a direct product of 
the nilpotent varieties of $ \lie{gl}_n $, 
nilpotent $ \bbK $-orbits in $ \Nilpotentvark $ are classified by 
pairs of partitions of size $ n $.  
Thus we identify $ \Nilpotentvark / \bbK \simeq \partitionsof{n}^2 $, and 
we write $ \Orbit_{(\lambda, \mu)} \; (\lambda, \mu \vdash n) $ 
for the corresponding nilpotent $ \bbK $-orbit.

\begin{lemma}
A nilpotent $ \bbK $-orbit $ \Orbit_{(\lambda, \mu)} \; (\lambda, \mu \vdash n) $ 
is $ \sigma $-stable if and only if $ \lambda = \mu $.  
In that case, the intersection is 
$ \Orbit_{(\lambda, \lambda)} \cap \nilpotentvark = \calorbit_{\lambda} $, 
which is the nilpotent $ K $-orbit with the Jordan normal form corresponding to $ \lambda $.
\end{lemma}

\begin{proof}
It is easy to see that $ \sigma( \Orbit_{(\lambda, \mu)} ) = \Orbit_{(\mu, \lambda)} $ 
by using the explicit form of the involution $ \sigma $.  
If $ \lambda = \mu $, the element 
$ \diag (x, - \transpose{x}) \in \Orbit_{(\lambda, \lambda)} $ belongs to $ \calorbit_{\lambda} $.
\end{proof}

Next we consider the nilpotent orbits in $ \Nilpotentvars $ and $ \nilpotentvars $.  
The nilpotent $ \bbK $-orbits in $ \Nilpotentvars $ are parametrized by 
signed Young diagrams of signature $ (n, n) $.  
We denote by $ \SYD(n,n) $ the set of all such signed Young diagrams.  
An element $ \Lambda \in \SYD(n,n) $ is a Young diagram of size $ 2n $ and 
each box is filled in by either plus or minus sign.  
In each row, $ (\pm) $-signs are arranged alternatively. The total number of plus signs is $ n $ and that of minus is also $ n $.
Moreover, the signed diagram is defined up to permutation of its rows.
We denote by $ \Orbit_{\Lambda} $ the nilpotent $ \bbK $-orbit corresponding to 
$ \Lambda \in \SYD(n,n) $.  

For example, the set $ \SYD(2,2) $ consists of 
\begin{equation*}
\ytableausetup
{boxsize=.5em}
\ytableausetup
{aligntableaux=top}
\begin{ytableau}
\smallplus & \smallminus & \smallplus & \smallminus 
\end{ytableau}
\;\;
\begin{ytableau}
\smallminus & \smallplus & \smallminus & \smallplus 
\end{ytableau}
\;\;
\begin{ytableau}
\smallplus & \smallminus & \smallplus 
\\
\smallminus
\end{ytableau}
\;\;
\begin{ytableau}
\smallminus & \smallplus & \smallminus 
\\
\smallplus 
\end{ytableau}
\;\;
\begin{ytableau}
\smallplus & \smallminus  
\\
\smallplus & \smallminus
\end{ytableau}
\;\;
\begin{ytableau}
\smallplus & \smallminus  
\\
\smallminus & \smallplus
\end{ytableau}
\;\;
\begin{ytableau}
\smallminus & \smallplus  
\\
\smallminus & \smallplus
\end{ytableau}
\;\;
\begin{ytableau}
\smallplus & \smallminus  
\\
\smallplus 
\\
\smallminus
\end{ytableau}
\;\;
\begin{ytableau}
\smallminus & \smallplus  
\\
\smallplus 
\\
\smallminus
\end{ytableau}
\;\;
\begin{ytableau}
\smallplus  
\\
\smallplus 
\\
\smallminus
\\
\smallminus
\end{ytableau}
\end{equation*}
and there are 10 orbits.  

On the other hand, 
nilpotent orbits in $ \nilpotentvars $ for type CI are parametrized by 
a subset of signed Young diagrams in $ \SYD(n,n) $ 
with the property that 

\medskip

\begin{itemize}
\item[$ (\star) $]
odd rows appear in pairs and their signature must be 
$ 
\ytableausetup
{aligntableaux=center}
\begin{ytableau}
\smallplus & \smallminus & \smallplus & \smallminus & \cdot & \cdot & \cdot & \smallminus & \smallplus
\\
\smallminus & \smallplus & \smallminus & \smallplus & \cdot & \cdot & \cdot & \smallplus & \smallminus
\end{ytableau}
$
\end{itemize}

\medskip

We denote the set of such signed Young diagrams by $ \SYDCI(n,n) $.  
For example, $ \SYDCI(2,2) $ consists of 
\begin{equation*}
\ytableausetup
{boxsize=.5em}
\ytableausetup
{aligntableaux=top}
\begin{ytableau}
\smallplus & \smallminus & \smallplus & \smallminus 
\end{ytableau}
\;\;
\begin{ytableau}
\smallminus & \smallplus & \smallminus & \smallplus 
\end{ytableau}
\;\;
\begin{ytableau}
\smallplus & \smallminus  
\\
\smallplus & \smallminus
\end{ytableau}
\;\;
\begin{ytableau}
\smallplus & \smallminus  
\\
\smallminus & \smallplus
\end{ytableau}
\;\;
\begin{ytableau}
\smallminus & \smallplus  
\\
\smallminus & \smallplus
\end{ytableau}
\;\;
\begin{ytableau}
\smallplus & \smallminus  
\\
\smallplus 
\\
\smallminus
\end{ytableau}
\;\;
\begin{ytableau}
\smallminus & \smallplus  
\\
\smallplus 
\\
\smallminus
\end{ytableau}
\;\;
\begin{ytableau}
\smallplus  
\\
\smallplus 
\\
\smallminus
\\
\smallminus
\end{ytableau}
\end{equation*}
so there are 8 orbits in total.  
We will denote the nilpotent $ K $-orbit corresponding to $ \Lambda \in \SYDCI(n,n) $ by 
$ \calorbit_{\Lambda} $.

\begin{lemma}
A nilpotent $ \bbK $-orbit $ \Orbit_{\Lambda} $ for $ \Lambda \in \SYD(n,n) $ 
is $ \sigma $-stable if and only if 
$ \Lambda $ belongs to $ \SYDCI(n,n) $.  
In that case, $ \Orbit_{\Lambda} \cap \nilpotentvars = \calorbit_{\Lambda} $.  
\end{lemma}

\begin{proof}
In fact, for any $ \Lambda \in \SYD(n,n) $, 
the involution takes $ \Orbit_{\Lambda} $ to 
$ \Orbit_{\sigma(\Lambda)} $, 
where $ \sigma(\Lambda) $ is obtained from 
$ \Lambda $ by exchanging plus and minus signs in the rows of \emph{odd length}.  
This follows from the consideration below.  
Let us take a representative $ x = \mattwo{0}{z}{w}{0} \in \Orbit_{\Lambda} $.  
Then the shape of $ \Lambda $ is given by the sizes of the Jordan cells in the Jordan normal form of $ x $.  
The signatures are controlled by 
\begin{equation*}
\begin{cases}
\rank (z w)^k
\\
\rank (w z)^k
\end{cases}
\quad
\text{ and }
\quad 
\begin{cases}
\rank (z w)^k z
\\
\rank (w z)^k w
\end{cases} 
\quad 
\text{ for $ k \geq 1 $}
\end{equation*}
(see \cite{Collingwood.McGovern.1993} and \cite[\S~10.1]{Fresse.N.2020} for details), for instance
\begin{align*}
\rank (wz)^k - \rank (zw)^k & = \#\{\text{odd rows of length $<2k$ starting with ``$+$''}\} \\
& \qquad
- \#\{\text{odd rows of length $<2k$ starting with ``$-$''}\}.
\end{align*}
By a direct calculation, 
we get 
\begin{equation*}
\sigma(x) = J_n^{-1} (- \transpose{x}) J_n 
= \mattwo{0}{\transpose{z}}{\transpose{w}}{0} .
\end{equation*}
Hence the signature of $ \sigma(\Lambda) $ is controlled by 
\begin{align*}
&
\begin{cases}
\rank (\transpose{z} \transpose{w})^k = \rank (w z)^k
\\
\rank (\transpose{w} \transpose{z})^k = \rank (z w)^k 
\end{cases}
\\
\intertext{ and }
&
\begin{cases}
\rank (\transpose{z} \transpose{w})^k \transpose{z} = \rank z (w z)^k = \rank (z w)^k z
\\
\rank (\transpose{w} \transpose{z})^k \transpose{w} = \rank w (z w)^k = \rank (w z)^k w.
\end{cases}
\end{align*}
Since the roles of $ (w z)^k $ and $ (z w)^k $ are exchanged, it affects the signature in the odd rows, and 
in fact, it exchanges the signature of plus and minus.

Thus, if $ \Orbit_{\Lambda} $ is $ \sigma $-stable, the odd rows must appear in pair and 
their signatures are exchanged.  Thus we have odd rows of the form appearing in ($\star$).  
The rest of the statements are clear.
\end{proof}

\medskip

Theorem \ref{thm:ohta} gives 
the orbit embeddings for nilpotent orbits $\iota^{\pm\theta}:\nilpotents^{\pm\theta}/K\hookrightarrow \Nilpotents^{\pm\theta}/\bbK$ 
in addition to the orbit embedding for orbits in double flag varieties
$\iota:\dblFV/K\hookrightarrow \bbdblFV/\bbK$ 
proved in Theorem \ref{thm:abstract.embedding.typeCI.to.typeAIII}.
However, in spite of the explicit description of all these orbit embeddings, 
the commutativity of the diagram \eqref{eq:embedding.and.Steinberg.maps.commute} is still open.


\bibliographystyle{amsplain}
%
\def\cftil#1{\ifmmode\setbox7\hbox{$\accent"5E#1$}\else
  \setbox7\hbox{\accent"5E#1}\penalty 10000\relax\fi\raise 1\ht7
  \hbox{\lower1.15ex\hbox to 1\wd7{\hss\accent"7E\hss}}\penalty 10000
  \hskip-1\wd7\penalty 10000\box7} \def\cprime{$'$} \def\cprime{$'$}
  \def\Dbar{\leavevmode\lower.6ex\hbox to 0pt{\hskip-.23ex \accent"16\hss}D}
\providecommand{\bysame}{\leavevmode\hbox to3em{\hrulefill}\thinspace}
\providecommand{\MR}{\relax\ifhmode\unskip\space\fi MR }
\providecommand{\MRhref}[2]{%
  \href{http://www.ams.org/mathscinet-getitem?mr=#1}{#2}
}
\providecommand{\href}[2]{#2}
\renewcommand{\MR}[1]{}

\end{document}